\documentclass[12pt]{article}
\usepackage{amsxtra,amssymb,amsthm,amsmath,latexsym,verbatim,epsfig}

\theoremstyle{plain}
\newtheorem{definition}{Definition}
\newtheorem{theorem}{Theorem}

\newtheorem{lemma}{Lemma}
\newtheorem{corollary}{Corollary}

\newcommand{\refT}[1]{Theorem~\ref{#1}}
\newcommand{\refS}[1]{Section~\ref{#1}}

\newcommand{\refL}[1]{Lemma~\ref{#1}}

\newcommand{\refC}[1]{Corollary~\ref{#1}}

\newcommand{\refD}[1]{Definition~\ref{#1}}

\newcommand{\refF}[1]{Fig.~\ref{#1}}

\newcommand{\stack}[2]{\genfrac{}{}{0pt}{}{#1}{#2}}

\def\R{{\mathbb R}}
\def\Q{{\mathbb Q}}

\def\Z{{\mathbb Z}}
\def\C{{\mathbb C}}
\def\D{{\mathbb D}}

\def\CP{\mathbb{CP}}

\def\calF{{\mathcal F}}
\def\calD{{\mathcal D}}
\def\P{\mathcal{P}}
\def\calW{\mathcal{W}}
\def\G{\Gamma}
\def\k{\varkappa}

\def\calE{{\mathcal E}}

\def\M{{\rm M}}

\def\gotM{{\mathfrak M}}
\def\Mbar{{\overline{M}}}

\def\gl{\mathfrak{gl}}
\def\Sl{\mathfrak{sl}}
\def\calO{{\mathcal O}}

\def\oH{\buildrel\circ\over H}
\def\oH1{\buildrel\circ\over H\kern-.02in{}^1}
\def\hookuparrow{{\cup\kern-.04in{}^\uparrow}}

\def\Re{\mathop{\rm Re}\nolimits}
\def\Res{\mathop{\rm Res}}

\def\Li{\mathop{\rm Li}\nolimits}
\def\csch{\mathop{\rm csch}\nolimits}
\def\vir{{\rm vir}}
\def\ev{{\rm ev}}

\def\bee{\begin{equation*}}
\def\eee{\end{equation*}}
\def\be{\begin{equation}}
\def\ee{\end{equation}}

\begin{document}

\title{Quantum Barnes function as the partition function of the resolved conifold}

\author{Sergiy Koshkin\\
Department of Mathematics,\\
Northwestern University, Evanston, IL 60208 USA\\
email: koshkin@math.northwestern.edu}

\date{}

\maketitle\thispagestyle{empty}

\begin{abstract}
We suggest a new strategy for proving large $N$ duality by interpreting Gromov-Witten, Donaldson-Thomas and Chern-Simons invariants of a Calabi-Yau threefold as different characterizations of the same holomorphic function. For the resolved conifold this function turns out to be the quantum Barnes function, a natural $q$-deformation of the classical one that in its turn generalizes Euler's gamma function. Our reasoning is based on a new formula for this function that expresses it as a graded product of $q$-shifted multifactorials. 
\end{abstract}

\section*{Introduction}

What is the topological string partition function of the resolved conifold\,? We should explain that heuristically one can assign string theories to each Calabi-Yau threefold and some of them such as {\it topological A-models} \cite{Wc}, only depend on its K\"ahler structure. Their topologically invariant amplitudes are then collected into a generating function called the partition function. Remarkably, this partition function may remain unchanged even if a threefold undergoes a topology changing transition \cite{GV}.

A traditional approach is to interpret the string partition function as the Gromov-Witten partition function. For the resolved conifold $X:=\calO(-1)\oplus\calO(-1)$ it was originally computed by Faber-Pandharipande \cite{FP}, see also \cite{Zhc}
\be\label{FaPan}
Z_X'(a;q)=\prod_{n=1}^\infty(1-aq^n)^n.
\ee
Here $a=e^{-t},\,q=e^{ix}$ and $t,x$ are known as the K\"ahler parameter and the string coupling constant respectively. In mathematical terms, they are just formal variables and 
\be\label{lnFaPan}
\ln Z_X'=\sum_{g=0}^\infty\sum_{d=1}^\infty\langle 1\rangle_{g,d}\ \,t^d\,x^{2g-2},
\ee
where $\langle 1\rangle_{g,d}$ is the Gromov-Witten invariant of genus $g$ degree $d$ holomorphic curves in the resolved conifold. 

The incompleteness of this answer does not reveal itself until one considers dualities that relate Gromov-Witten invariants to other invariants of Calabi-Yau threefolds. One may notice that 
\eqref{lnFaPan} is missing degree zero terms (hence the $'$\,). This is not a slip, they can not be packaged into a form as nice as \eqref{FaPan}. This was not considered much of a problem until Donaldson-Thomas theory \cite{Ktz,MNOP} came about since degree zero (constant) maps are trivial anyway. But apparently dualities have little tolerance for convenient omissions. For Gromov-Witten/Donaldson-Thomas duality to hold \eqref{FaPan} has to be augmented as 
\be\label{totFaPan}
Z_X=Z_X^0\,Z_X'\approx\M\,Z_X',
\ee
where 
\be\label{McM}
\M(q):=\prod_{n=1}^\infty(1-q^n)^{-n}
\ee
is the MacMahon function, classically known as the generating function of plane partitions \cite{Mc}. In all honesty, this is not quite true as $\ln\M(e^{ix})$ has some spurious terms in its expansion at $x=0$ and only accounts for genus $g\geq2$ terms correctly (see \refS{S3}). Also in Donaldson-Thomas theory one has $Z_X^{DT}=\M^2\,Z^{\prime\,DT}_X$, not $Z_X=\M\,Z_X'$. In a recent reformulation of the Donaldson-Thomas theory \cite{PT} the {\it reduced partition function} $Z^{\prime\,DT}$ is even defined directly and the Mac-Mahon function is banished altogether. Let us disregard this minor discrepancy for now since even answer \eqref{totFaPan} is incomplete.

This becomes apparent in light of another duality of Calabi-Yau threefolds, large $N$ duality. This one relates Gromov-Witten invariants of the resolved conifold to Chern-Simons invariants of the $3$-sphere. The usual formulation defines 
Chern-Simons theory as a gauge theory on a $U_N$ or $SU_N$ bundle over a real $3$-manifold $M$. Less recognized despite Witten's famous paper \cite{Wc} is the fact that it also gives invariants of Calabi-Yau threefolds. As Witten pointed out it can be viewed as a theory of open strings (holomorphic instantons at $\infty$ in his terminology) in the cotangent bundle $T^*M$ ending on its zero section. $T^*M$ is canonically a symplectic manifold (even K\"ahler if $M$ is real-analytic) with first Chern class $c_1(T^*M)=0$, i.e. a Calabi-Yau. In particular, $T^*S^3$ is diffeomorphic to a quadric $x^2+y^2+z^2+w^2=1$ in $\C^4$. One of the reasons this interpretation did not get much currency is that the strings in question are very degenerate, they are represented by ribbon graphs, not honest holomorphic curves. In fact, there are no honest holomorphic curves in $T^*M$ at all except for the constant ones \cite{K,Wc}. Another reason perhaps, is that open Gromov-Witten theory is still in its infancy and the powerful algebro-geometric techniques that dominate the field can not be directly applied. There are successful approaches that replace open invariants with relative ones \cite{BP,L3Z} but only as a tool for computing closed invariants. In the other direction, there exists a detailed if only formal correspondence between geometry of real oriented $3$-manifolds and Calabi-Yau threefolds and Donaldson-Thomas theory can be seen as a 'holomorphization' of Chern-Simons theory under this correspondence \cite{Ty}. Thus, comparing the Chern-Simons partition function $Z_{S^3}$ to $Z_X$ promises some useful insight. 

Once again ignoring some irrelevant prefactors, $Z_{S^3}$ can be written as $Z_{S^3}\approx\calE^{-z}\,Z_X$, where $z=itx^{-1}$ so that $a=q^z$ and 
\be\label{Eul}
\calE(q):=\prod_{n=1}^\infty(1-q^n)^{-1}
\ee
is the classical Euler generating function of ordinary partitions. At this point it is appropriate to introduce notation that allows one to write $Z_X',\M$ and $\calE$ uniformly. Let 
\be\label{qfac}
(a;q)_\infty^{(0)}:=1-a,\quad(a;q)_\infty^{(d)}:=\!\!\!\!\prod_{i_1,\dots,i_d=0}^\infty(1-aq^{i_1+\cdots+i_d})
\ee
be the {\it $q$-multifactorials} then (see \refS{S6})
$$
Z_X'(a;q)=(aq;q)_\infty^{(2)},\quad\M(q)=\frac1{(q;q)_\infty^{(2)}},\quad\calE(q)=\frac1{(q;q)_\infty^{(1)}}
$$
Using $q$ and $z$ as variables we see that 
\begin{align}\label{lineup}
Z_X' &=(q^{z+1};q)_\infty^{(2)}\notag\\
Z_X &\approx\frac1{(q;q)_\infty^{(2)}}\ \ (q^{z+1};q)_\infty^{(2)}\\
Z_{S^3} &\approx(q;q)_\infty^{(1)\,\,z}\ \frac1{(q;q)_\infty^{(2)}}\ \ (q^{z+1};q)_\infty^{(2)}\notag
\end{align}
After some thought one may sense a pattern here. We shall see in \refS{S6} that it makes sense to join one more factor to the product and consider
\be\label{entGq}
G_q(z+1):=\frac1{(q;q)_\infty^{(0)\ \frac{z(z-1)}2}}\ (q;q)_\infty^{(1)\,\,z}
\ \frac1{(q;q)_\infty^{(2)}}\ \ (q^{z+1};q)_\infty^{(2)}.
\ee
This $G_q$ is the {\it quantum Barnes function} of Nishizawa \cite{Ni1} and our candidate for the partition function of the resolved conifold. {\it All} factors above are required to make it transform as
\be\label{propGq}
G_q(z+1)=\G_q(z)\,G_q(z),
\ee
where $\G_q$ is Jackson's {\it quantum gamma function} deforming the classical one. This in turn satifies $\G_q(z+1)=(z)_q\,\G_q(z)$ with the so-called {\it quantum number} $(z)_q:=\frac{1-q^z}{1-q}$. This makes $G_q$ a deformation of the {\it classical Barnes function} that satifies \eqref{propGq} with $q$-s removed. 

The picture above is cute but not quite true and clear-cut identities \eqref{lineup} are spoiled by pesky disturbances discussed in Sections \ref{S3} and \ref{S5}. These disturbances are a large part of the reason why large $N$ duality is so hard to prove even in simple cases. Still, $G_q$ emerges as a common factor in Gromov-Witten, Donaldson-Thomas and Chern-Simons theories (\refT{Comparison}). One may notice that we conspicuously omitted the most famous of Calabi-Yau dualities, mirror symmetry. This is partly because local mirror symmetry is poorly developed, and partly because to the extent that its predictions can be divined \cite{FJ} they match the Gromov-Witten ones completely. There is a structural prediction of mirror symmetry that seems relevant. For {\it compact} Calabi-Yau threefolds $Z$ is predicted to have modular properties \cite{Dj}, i.e. obey transformation laws under $z\mapsto z+1$ and $z\mapsto-1/z$. For {\it open} threefolds like the resolved conifold only the first one survives and is expressed by \eqref{propGq}.

What are we to make of the above chain of augmentations? Perhaps, string theories on Calabi-Yau threefolds are only partial reflections of some hidden {\it master-theory}. Witten's candidate for such a theory is the mysterious M-theory living on a seven-dimensional manifold with $G_2$ holonomy that projects to various string theories on Calabi-Yau threefolds. Another  unifying view of Gromov-Witten and Donaldson-Thomas theories, via non-commutative geometry, also emerged recently \cite{Sz}. Different projections are equivalent even though they may live on topologically distinct threefolds and reflect the original each in their own way. So far we ignored these ways relying instead on magical changes of variables. It is time to dwell upon them a bit. This will also serve as our justification for spending so much ink on the resolved conifold.

The relation between Gromov-Witten and Donaldson-Thomas invariants is very simple \cite{MNOP}. For the resolved conifold we have 
\be\label{lnDT}
\ln Z_X'=\sum_{n=0}^\infty\sum_{d=1}^\infty\,(-1)^nD_{n,d}\ \,t^d\,q^n 
\ee
with $Z_X'$ the same as in \eqref{lnFaPan} and $D_{n,d}$ the Donaldson-Thomas invariants. In other words, in each degree $(-1)^nD_{n,d}$ are simply the Taylor coefficients of $Z_X'$ at $q=0$ while $\langle 1\rangle_{g,d}$ are the Laurent coefficients in $x$ corresponding to $q=1$ with $q=e^{ix}$. The relation with Chern-Simons invariants is more complicated. Traditionally, one has to take $q=e^{2\pi i/(k+N)}$, where $k,N$ are the two parameters of Chern-Simons theory, rank and level. They are positive integers making $q$ a root of unity. Not all roots of unity are covered in this way but more sophisticated formulations allow one to include any root of unity. Naively, if the duality conjectures hold Donaldson-Thomas invariants give us an expansion at $q=0$, Gromov-Witten invariants at $q=1$ and Chern-Simons invariants give values at roots of unity of more or less the {\it same function}. But only naively. First of all, Donaldson-Thomas generating functions are a priori only formal power series and may not have a positive radius of convergence. We need it to be at least $1$ to make a comparison. Things are nice in higher degree \cite{PT} but {\it in degree zero} it is exactly $1$ and every point of the unit circle is a singularity. This is remedied easily enough in the Gromov-Witten context since we can interpret $Z_X^0(e^{ix})=\sum_{g=0}^\infty\langle 1\rangle_{g,0}\,x^{2g-2}$ as an {\it asymptotic} expansion at the natural boundary (\refS{S3}). But Chern-Simons invariants are not graded by degree and the degree zero speck turns into a wooden beam spoiling the whole partition function that we wish to evaluate. With the resolved conifold being the simplest non-trivial case we get a preview of the difficulties that will arise in general. This brings us to a paradox: for large $N$ duality to even make sense the formal power series better converge to holomorphic functions extending to the unit circle or at least to roots of unity. This is {\it not} the case already for $Z_X$ and an additional factor in $Z_{S^3}$ appearing in \eqref{lineup} is needed to make it happen (see comments after \refC{Nniprod}). This is another reason to accept the quantum Barnes function as the completed partition function. 

Since conjecturally $Z_Y^0=\M^{\frac12\chi(Y)}$ for any Calabi-Yau threefold $Y$ \cite{MNOP} this phenomenon is likely to be general. The above discussion suggests that the {\it master-invariant} that manifests itself through dualities is a holomorphic function on the unit disk. The three theories we discussed showcase three different ways to package information about it. The dualities reduce to repackaging prescriptions. Physicists developed resummation techniques that transform generating functions one into another but they lead to unwieldy computations for the resolved conifold and do not produce conclusive results even for its cyclic quotients \cite{AKMVm}. Since repackaging involves transcendental substitutions, analytic continuation and asymptotic expansions -- things one does with functions and not with formal series -- it makes sense to identify the underlying holomorphic functions to establish a duality. This is the strategy of this paper and it distinguishes it from previous approaches \cite{AK,GV,Mar2} that use double expansions in genus and degree. This makes for a cleaner comparison of partition functions with a clear view of what matches and what {\it does not match} in them (\refT{Comparison}). It is also hoped that the idea generalizes to other threefolds.

The paper is organized as follows. \refS{S2} is a review of basic notions of Gromov-Witten theory with emphasis on generating functions. In particular, we note that free energy is a shorthand for the Gromov-Witten potential restricted to divisor invariants. The well-known irregularities in degree zero are then naturally explained. In \refS{S3} the MacMahon function is examined in detail to determine to what extent it can be viewed as the degree zero partition function of Gromov-Witten invariants. We describe resummation techniques used by physicists, and then recall an old but little known asymptotic for it due to Ramanujan and Wright adapting it to our context. Sections \ref{S4} and \ref{S5} give a description of the topological vertex and the Reshetikhin-Turaev calculus, diagrammatic models that compute Gromov-Witten and Chern-Simons partition functions respectively. Similarities between the two are specifically stressed. \refS{S5} ends by expressing both partition functions via the quantum Barnes function (\refT{Comparison}). Since this function and its higher analogs are relatively recent (1995) we give a self-contained exposition of their theory in \refS{S6} different from the author's \cite{Ni1}. In particular, we prove the  alternating formula \eqref{entGq} that connects $G_q$ to the Calabi-Yau partition functions and appears to be new (\refT{Altz}). In Conclusions we point out relations between Calabi-Yau dualities and holography and share some thoughts and conjectures inspired by the resolved conifold example. The Appendix lists basic properties of the Stirling polynomials needed in \refS{S6}.

\smallskip

{\bf Acknowledgements.} The author wishes to thank D.Auckly and D.Karp for sharing their thoughts and notes on Gromov-Witten and Chern-Simons theories, and suggesting numerous corrections and improvements to the original draft. I am also indebted to M.Mari\~no for his extensive email comments on the physical background of Calabi-Yau dualities. His influence can be felt in the overall mindset of this paper.

\section{Generating functions of\\ Gromov-Witten invariants}\label{S2}

There are a variety of generating functions appearing in the literature: Gromov-Witten potential, prepotential, truncated potential, partition function, free energy, etc. In this section we briefly review basic definitions from Gromov-Witten theory and relationships among some of the above generating functions. Perhaps, the only unconventional notion is that of divisor potential which leads most naturally to the free energy and the partition function. 

\smallskip

\noindent {\bf Stable maps.}\quad Let $X$ be a K\"ahler manifold of complex dimension $N$. We wish to consider holomorphic maps ${f\!:\ }\Sigma\to X$ of Riemann surfaces with $n$ marked points into $X$ that realize certain homology class $\alpha\in H_2(X,\Z)$. The space of such maps is denoted $M_{g,n}(X,\alpha)$. There is a natural (Gromov) topology on this moduli space but it is not compact in it. To get Gromov-Witten invariants we need to integrate over the moduli so we have to compactify. The appropriate compactification was discovered by Kontsevich and its elements are called stable maps. They are holomorphic maps from {\it prestable curves}, i.e. connected reduced projective curves with at worst ordinary double points (nodes) as singularities. A map is {\it stable} if its group of automorphisms is finite, i.e. there are only finitely many biholomorphisms $\sigma:\Sigma\to\Sigma$ satisfying $f\circ\sigma=f$ and $\sigma(p_i)=p_i$, where $p_1\dots,p_n$ are the marked points. Intuitively, we allow Riemann surfaces to degenerate by collapsing loops into points. Since only genus $0$ and $1$ curves have infinitely many automorphisms (M\"obius transformations and translations respectively) the stability condition is non-vacuous only for them and only if the map $f$ is trivial, i.e. maps everything into a point. It requires then that each genus $0(1)$ component have at least $3(1)$ special points, nodes or marked points. Under favorable circumstances the space of stable maps $\Mbar_{g,n}(X,\alpha)$ up to reparametrization is itself a closed K\"ahler orbifold of dimension  
\be\label{dimvir}
\dim_\C^\vir\Mbar_{g,n}(X,\alpha)=\langle c_1(X),\alpha\rangle-(N-3)(g-1)+n.
\ee
For instance, this is the case if $X=\CP^N$ and $g=0$. Above $c_1(X)$ is the first Chern class of the tangent bundle and $\langle,\rangle$ the cohomology/homology pairing. The notation anticipates that in general the moduli is neither smooth nor has the expected dimension so \eqref{dimvir} is called the {\it virtual dimension}. A deep result in Gromov-Witten theory asserts that despite the complications there is a cycle of expected dimension $[\Mbar_{g,n}(X,\alpha)]^\vir$ called the {\it virtual fundamental class} that one can integrate over.

\smallskip

\noindent {\bf Primary invariants.}\quad Presence of marked points allows one to define evaluation maps 
$$
\begin{aligned}
  \ev_i:\Mbar_{g,n}(X,\alpha) & \to\quad X\\
  f\qquad &\longmapsto f(p_i)
 \end{aligned}
$$
and pullback cohomology classes $\gamma_i$ from $X$ to $\Mbar_{g,n}(X,\alpha)$. These pullbacks are called the {\it primary classes} on $\Mbar_{g,n}(X,\alpha)$ \cite{Gz,KM}. The {\it primary Gromov-Witten invariants} are 
\be\label{prim}
\langle \gamma_1\cdots\gamma_n\rangle_{g,\alpha}:=\int_{[\Mbar_{g,n}(X,\alpha)]^\vir}\,
\ev_1^*(\gamma_1)\cup\cdots\cup\ev_n^*(\gamma_n),
\ee
where $\cup$ is the usual cup product and the integral denotes pairing with $[\Mbar_{g,n}(X,\alpha)]^\vir$. Again under favorable circunstances the primary invariants have an enumerative interpretation. Namely, $\langle \gamma_1\cdots\gamma_n\rangle_{g,\alpha}$ is the number of genus $g$ holomorphic curves in a class $\alpha\in H_2(X,\Z)$ passing through generic representatives of cycles Poincare dual to $\gamma_1,\dots,\gamma_n$ \cite{AK,KVq}. In general, the enumerative interpretation fails and $\langle \gamma_1\cdots\gamma_n\rangle_{g,\alpha}$ are only rational numbers, this is always the case for Calabi-Yau manifolds. Most of the primary invariants are zero for dimensional reasons. Indeed, the complex degree of the integrand in \eqref{prim} is $\frac12(\deg\gamma_1+\cdots+\deg\gamma_n)$ and for the integral to be non-zero it should be equal to the virtual dimension \eqref{dimvir}. There are other natural classes on $\Mbar_{g,n}(X,\alpha)$ that lead to more general Gromov-Witten invariants, gravitational descendants and Hodge integrals \cite{AK,FP,Gz}, but we need not concern ourselves with them here.

It is convenient to arrange the primary invariants into a generating function \cite{KVq}. To this end we note that they are linear in insertions $\gamma_i$ and we can recover all of them from $\langle1\rangle_{g,\alpha}$ and those with insertions chosen from an integral basis $h_1,\dots,h_m$ in $H^+(X,\Z):=\oplus_{n>0}H^n(X,\Z)$. One may worry about torsion but torsion classes are not represented by holomorphic curves and can be ignored. Thus any $\langle \gamma_1\cdots\gamma_n\rangle_{g,\alpha}$ is a linear combination of $\langle h_1^{p_1}\cdots h_m^{p_m} \rangle_{g,\alpha}$\ , where the 'powers' $p_i$ stand for repeating $h_i$ that many times. Introduce formal variables $t_1,\dots,t_m$ for each element of the basis. Heuristically, they represent (minus) K\"ahler volumes of $h_1,\dots,h_m$ and are called {\it K\"ahler parameters,} especially in physics literature. Analogously, let $\xi_1,\dots,\xi_k$ be a linear basis in $H_2(X,\Z)$ and $Q_1,\dots,Q_k$ the corresponding formal variables. We write $\langle h_1^{p_1}\cdots h_m^{p_m} \rangle_{g,\vec{d}}$\ \ with $\vec{d}:=(d_1,\dots,d_k)$ for short when $\alpha=d_1\xi_1+\cdots+d_k\xi_k$. The numbers $d_1,\dots d_k$ are called {\it degrees}. Finally, we need one more variable $x$, the {\it string coupling constant}, to incorporate genus. The {\it primary Gromov-Witten potential} (relative to the above bases choices) is 
\be\label{primpot}
\calF(t_1,\dots,t_m;Q_1,\dots,Q_k;x):=\sum_{g=0}^\infty\sum_{\stack{p_1,\dots,p_m=0}{d_1,\dots,d_k=0}}^\infty\!\!\!
\langle h_1^{p_1}\cdots h_m^{p_m} \rangle_{g,\vec{d}}\ \frac{t_1^{p_1}\cdots t_m^{p_m}}{p_1!\cdots p_m!}\,
Q_1^{d_1}\cdots Q_k^{d_k}\,x^{2g-2}. 
\ee
This particular choice of a generating function is by no means obvious and is inspired by two-dimensional topological quantum field theory. The power $2g-2$ instead of just $g$ has in mind the Euler characteristic $-(2g-2)$ of a genus $g$ Riemann surface. For $X$ K\"ahler $\calF$ is defined as at least a formal power series in $\Q[[t_j,Q_i,x]]$ \cite{Gv2}. Under a change of bases $\langle h_1^{p_1}\cdots h_m^{p_m} \rangle_{g,\vec{d}}$ transforms as a tensor. One may entertain oneself by writing a tensor potential that is an invariant, see \cite{KVq}. In \cite{Gz,KM} a more general Gromov-Witten potential is considered that incorporates gravitational descendants and accordingly has more formal variables.

Let $X:=\calO(-1)\oplus\calO(-1)$ be the resolved conifold, the sum of two tautological line bundles over $\CP^1$. Being a vector bundle over $\CP^1$ it is homotopic to its base and has the same homology and cohomology. In particular, $H_2(X,\Z)=\Z[\CP^1]$ and $H^\bullet(X,\Z)=\Z[h]/(h^2)$, where $h$ is the Poincare dual to the class of a point in $\CP^1$. Thus, $H^+(X,\Z)$ is spanned by $h$ and $H_2(X,\Z)$ is spanned by $\xi:=[\CP^1]$, the fundamental class of $\CP^1$. Hence we need only one $t$ and one $Q$ variable. The primary potential simplifies to 
\be\label{resconpot}
\calF(t;Q;x):=\sum_{p,d,g=0}^\infty \langle h^p\rangle_{g,d}\ \frac{t^p}{p!}\,Q^d\,x^{2g-2}. 
\ee

\smallskip

\noindent {\bf Divisor equation and free energy.}\quad We will be interested not even in all primary invariants but only in those corresponding to combinations of {\it divisor classes}, elements of $H^2(X,\Z)$. Divisor invariants turn out to be most relevant to large $N$ duality. In non-compact manifolds the name is misleading since there is no Poincare duality. For example, the hyperplane class of $\CP^1$ is a divisor class in $\calO(-1)\oplus\calO(-1)$ despite the fact that it is not dual to any divisor. But in closed manifolds divisor classes are precisely Poincare duals to divisors, cycles of complex codimension one.  Invariants $\langle h_1^{p_1}\cdots h_m^{p_m} \rangle_{g,\vec{d}}$\ \ containing only divisor classes can be reduced to $\langle1\rangle_{g,\vec{d}}$ \ \ using the so-called divisor equation. The latter is one of the universal relations among Gromov-Witten invariants coming from universal relations among moduli spaces of stable maps with the same target $X$. One of them is \cite{Gz,KM}
\be\label{forget}
\pi^*[\Mbar_{g,n-1}(X,\alpha)]^\vir=[\Mbar_{g,n}(X,\alpha)]^\vir,
\ee
where $\Mbar_{g,n}(X,\alpha)\xrightarrow[\pi]{}\Mbar_{g,n-1}(X,\alpha)$ is the map forgetting the last marked point. Its consequence is the {\it divisor equation} 
\be\label{div}
\langle h\,\gamma_1\cdots\gamma_n\rangle_{g,\alpha}=h(\alpha)\langle \gamma_1\cdots\gamma_n\rangle_{g,\alpha},
\ee
where $h\in H^2(X,\Z)$ and $\gamma_i$ are arbitrary. There are two exceptions to the validity of \eqref{forget} and hence \eqref{div}, both in degree zero. If $\alpha=0$ then $\Mbar_{g,n}(X,0)$ consists of constant maps. The stability condition requires domains of stable maps in this case to be themselves stable, not just prestable. But when $g=0(1)$ a stable curve must have at least $3(1)$ marked points so the spaces of curves $\Mbar_{0,0}, \Mbar_{0,1}, \Mbar_{0,2}, \Mbar_{1,0}$ are empty. However, $\Mbar_{0,3}, \Mbar_{1,1}$ are not and \eqref{forget} fails for $(g,n)=(0,3),(1,1)$. 

Since $H^2(X,\Z)\simeq H_2(X,\Z)$ modulo torsion and $\xi_1,\dots,\xi_k$ form a basis in $H_2(X,\Z)$ there are precisely $k$ basis elements in $H^2(X,\Z)$. We assume without loss of generality that $h_1,\dots,h_k$ are the ones and that they are dual to $\xi_1,\dots,\xi_k$, i.e. $h_i(\xi_j)=\delta_{ij}$. The divisor equation may now be used to flush all the insertions out of the divisor invariants. By induction from \eqref{div}
\be\label{flush}
\langle h_1^{p_1}\cdots h_k^{p_k}\rangle_{g,\vec{d}}=
d_1^{p_1}\cdots d_k^{p_k}\,\langle1\rangle_{g,\vec{d}}\ ,
\ee
assuming $\vec{d}\neq0$ to avoid low-genus problems in degree zero. Define the {\it truncated divisor potential} $\calF_{div}'(t_1,\dots,t_k;Q_1,\dots,Q_k;x)$ as in \eqref{primpot} but restricting the sum to $p_1,\dots,p_k$ and $\vec{d}\neq0$. Using \eqref{flush} we compute
\begin{multline*}
\calF_{div}'(t_1,\dots,t_k;Q_1,\dots,Q_k;x) =\sum_{g=0}^\infty\sum_{\vec{d}\neq0}\langle 1\rangle_{g,\vec{d}}\ 
Q_1^{d_1}\cdots Q_k^{d_k}\,x^{2g-2}\!\!\!\!\!\!
\sum_{p_1,\dots,p_k=0}^\infty\!\!\!\!\!\frac{(d_1t_1)^{p_1}\cdots (d_mt_m)^{p_m}}{p_1!\cdots p_m!}\\
=\!\sum_{g=0}^\infty\sum_{\vec{d}\neq0}\langle 1\rangle_{g,\vec{d}}\ 
Q_1^{d_1}\cdots Q_k^{d_k}\,x^{2g-2}\,e^{d_1t_1}\cdots e^{d_kt_k}\!
=\!\!\sum_{g=0}^\infty\sum_{\vec{d}\neq0}\langle 1\rangle_{g,\vec{d}}\ 
(Q_1e^{t_1})^{d_1}\cdots (Q_ke^{t_k})^{d_k}\,x^{2g-2}\!\!.
\end{multline*}
Obviously, as far as divisor invariants go $Q_1,\dots,Q_k$ are redundant and we can set them equal to $1$. This naturally leads to another generating function \cite{KL,MNOP}:
\begin{definition}\label{trunc}
The reduced Gromov-Witten free energy is 
\be\label{etrunc}
F'(t_1,\dots,t_k;x):=\sum_{g=0}^\infty\sum_{\vec{d}\neq0}\langle 1\rangle_{g,\vec{d}}\ 
e^{d_1t_1}\cdots e^{d_kt_k}\,x^{2g-2}. 
\ee
Its exponent $Z'(t_1,\dots,t_k;x):=\exp(F'(t_1,\dots,t_k;x))$ is called the reduced Gromov-Witten partition function. We write $F_X', Z_X'$ when the target manifold needs to be indicated.
\end{definition}
The reduced free energy is non-zero only if $\langle c_1(X),\alpha\rangle-(N-3)(g-1)=0$ for some class $\alpha\neq0$, see \eqref{dimvir}. If $X$ is a Calabi-Yau then $c_1(X)=0$ and if in addition it is a threefold then also $N=3$ and the non-triviality condition holds for all classes and genera. For a toric Calabi-Yau $X$ the reduced partition function $Z_X'$ is the quantity directly computed by the topological vertex algorithm \cite{AKMVv,Kh,L3Z,Mar2}.

\smallskip

\noindent {\bf Degree zero.}\quad The moduli spaces $\Mbar_{g,n}(X,0)$ consist of stable maps mapping stable curves into points. Therefore they split \cite{FP}
\be\label{split}
\Mbar_{g,n}(X,0)=\Mbar_{g,n}\times X.
\ee
This reduces degree zero invariants to integrals over the spaces of curves and over $X$. The divisor equation \eqref{div} still holds for $n\geq4(2)$ for genus $g=0(1)$ and for all $n$ in higher genus. Moreover, since $\alpha=0$ now it directly implies that all the divisor invariants vanish except possibly for those that can no longer be reduced. Therefore, in genus $g\geq2$ the only surviving invariants are $\langle 1\rangle_{g,0}$ and in genus $0,1$ we are left with $\langle h_i^3\rangle_{0,0}, \langle h_i^2h_j\rangle_{0,0}, \langle h_ih_jh_l\rangle_{0,0}$ and $\langle h_i\rangle_{1,0}$ respectively. There is automatically no dependence on $Q_i$ so the degree zero divisor potential is the same as the degree zero free energy (cf. \cite{P}):
\begin{multline}
F^0(t_1,\dots,t_k;x):=\calF^0(t_1,\dots,t_k;x)\\
=\left(\sum_{i=1}^k\langle h_i^3\rangle_{0,0}\,\frac{t_i^3}6+
\sum_{i\neq j}\langle h_i^2h_j\rangle_{0,0}\,\frac{t_i^2t_j}2+
\sum_{i\neq j,j\neq l,l\neq i}\langle h_ih_jh_l\rangle_{0,0}\ t_it_jt_l\right)\frac1{x^2}\\
+\sum_{i=1}^k\langle h_i\rangle_{1,0}\,t_i+\sum_{g=2}^\infty\langle1\rangle_{g,0}\,x^{2g-2}.
\end{multline}
Note that degree zero genus $0(1)$ terms are the only parts of the free energy depending on powers of $t_i$ rather than just their exponents $e^{t_i}$. When $X$ is {\it compact} these terms reflect its classical cohomology, namely \cite{P}  
\begin{align}\label{classic}
\langle h_ih_jh_l\rangle_{0,0} &=\int_X h_i\cup h_j\cup h_l \notag\\
\langle h_i\rangle_{1,0} &=-\frac1{24}\int_X h_i\cup c_2(X).
\end{align}
In particular, they vanish unless $X$ is a threefold. Higher genus contributions were computed in the celebrated paper of Faber-Pandharipande \cite{FP}
\begin{multline}\label{geq2}
\langle 1\rangle_{g,0}=\frac{(-1)^g\ \ |B_{2g}||B_{2g-2}|}{(2g-2)!\ 2g\,(2g-2)}\cdot\frac12\int_X (c_3(X)-c_1(X)\cup c_2(X))\\
=\frac{(-1)^{g-1}(2g-1)B_{2g}B_{2g-2}}{(2g-2)(2g)!}\cdot\frac12\int_X (c_3(X)-c_1(X)\cup c_2(X)),\quad g\geq2.
\end{multline}
Here $c_i(X)$ as before are Chern classes and $B_n$ are the Bernoulli numbers defined via a generating function \cite{AAR}
\be\label{Bern}
\frac{z}{e^z-1}=:\sum_{n=0}^\infty B_n\,\frac{z^n}{n!}. 
\ee
The only non-zero odd-indexed number is $B_1=-1/2$ and $B_0=1$, $B_2=1/6$, $B_4=-1/30$, $B_6=1/42$.

One sees from \eqref{geq2} that higher genus contributions all vanish for non-threefolds even when non-divisor invariants are taken into account because the Chern classes integrate to zero. However, genus $0(1)$ terms may still survive if $X$ has cohomology classes of appropriate degree to cup with $c_2(X)$ and each other. But the divisor invariants still vanish for dimensional reasons. Also note that \eqref{geq2} simplifies for Calabi-Yau threefolds since $c_1(X)=0$ and $\int_X c_3(X)=\chi(X)$ is the Euler characteristic of $X$. Thus, for compact Calabi-Yau threefolds
\be\label{CYgeq2}
\langle 1\rangle_{g,0}=\frac{(-1)^{g-1}(2g-1)B_{2g}B_{2g-2}}{(2g-2)(2g)!}\cdot\frac{\chi(X)}2,\quad g\geq2.
\ee

When $X$ is non-compact but $\alpha\neq0$ the moduli $\Mbar_{g,n}(X,\alpha)$ may still be compact. This usually happens if geometry forces images of stable maps to stay within a fixed compact subset of $X$, e.g. this is the case for the resolved conifold \cite{KL,K}. Then the virtual class is still defined and no new problems arise. However, if $\alpha=0$ factorization \eqref{split} forces $\Mbar_{g,n}(X,0)$ to be non-compact always. To the best of our knowledge no virtual class theory exists for non-compact moduli so technically $\langle \gamma_1\cdots\gamma_n\rangle_{g,0}$ for non-compact $X$ are not defined at all. 

Leaving the land of rigor and arguing like string theorists we notice that for Calabi-Yau threefolds \eqref{CYgeq2} still makes sense and can be taken as the 'right' answer even for non-compact $X$. This is consistent with a formal localization computation \cite{AK}. Unfortunately, for $g=0,1$ the invariants contain insertions and we really need to know how to interpret the integrals over $X$ in \eqref{classic}. In physics literature it is suggested that they correspond to integrals over "non-compact cycles" \cite{FJ} that can perhaps be interpreted as duals to compact cohomology cocycles \cite{BT}. We conclude that for the resolved conifold ($\chi(X)=2$) the degree zero free energy has the form
\be\label{F0rescon}
F^0_{\calO(-1)\oplus\calO(-1)}(t;x)=\frac{p_3(t)}{x^2}+p_1(t)+\sum_{g=2}^\infty (-1)^{g-1}\frac{(2g-1)B_{2g}B_{2g-2}}{(2g-2)(2g)!}\,x^{2g-2},
\ee
where $p_i$ are degree $i$ homogeneous polynomials with rational coefficients. We should mention that there are reasonable ways \cite{FJ} of assigning values to $p_3,p_1$ at least for local curves (see \cite{BP}) from equivariant and mirror symmetry viewpoints. For the resolved conifold they yield 
$$
F^0_{\calO(-1)\oplus\calO(-1)}(t;x)=\frac{t^3}{6}\frac{1}{x^2}+\frac{t}{12}+\sum_{g=2}^\infty (-1)^{g-1}\frac{(2g-1)B_{2g}B_{2g-2}}{(2g-2)(2g-2)!}\,x^{2g-2}
$$
and this function can be recovered from the mirror geometry. However, it appears that Donaldson-Thomas and Chern-Simons theories store classical cohomology information more crudely. We shall see that in genus $0,1$ this answer or even the general template \eqref{F0rescon} is {\it inconsistent} with exact duality (see discussion after \refC{MacAsIm}).
\begin{definition}
The (full) Gromov-Witten free energy is $F:=F^0+F'$ and the (full) Gromov-Witten partition function is $Z:=exp(F)=Z^0Z'$, where $F',Z'$ are reduced versions from \refD{trunc}. As before we write $F_X, Z_X$ to indicate the target manifold if necessary.
\end{definition}
For the resolved conifold we get from \eqref{etrunc}
\be\label{resconfren}
F_{\calO(-1)\oplus\calO(-1)}(t;x)=F^0_{\calO(-1)\oplus\calO(-1)}(t;x)+\sum_{\stack{g=0}{d=1}}^\infty \langle1\rangle_{g,d}\ e^{dt}\,x^{2g-2}.
\ee
The positive degree part converges to a holomorphic function in an appropriate domain of $t,x$ (recall that $t$ is a {\it negative} K\"ahler volume). The same holds for all toric Calabi-Yau threefolds and for them the partition function is given directly by the topological vertex \cite{AKMVv,Kh,L3Z,Mar2}. We will discuss the case of the resolved conifold in more detail in \refS{S4}. But the degree zero part is not so well behaved. The sum in \eqref{F0rescon} diverges and fast! By a classical estimate for Bernoulli numbers
\be\label{estBern}
\frac{(2g)!}{\pi^{2g}\,2^{2g-1}}<|B_{2g}|<\frac{(2g)!}{\pi^{2g}\,(2^{2g-1}-1)},\quad g\geq1,
\ee
and the general term in \eqref{F0rescon} grows factorially for any $x\neq0$. Coming up with a space of formal power series where the sum lives is neither difficult nor helpful. A helpful insight comes from the conjectural duality with  Donaldson-Thomas theory \cite{MNOP} that suggests to view \eqref{F0rescon} as an {\it asymptotic expansion} of a holomorphic function at a natural boundary point. The function in question is the Mac-Mahon function $\M(q)$, the point is $q=1$ and the relation to \eqref{F0rescon} is $q=e^{ix}$. We inspect this idea in the next section.

\section{Donaldson-Thomas theory\\ and the Mac-Mahon function}\label{S3}

In this section we clarify the relationship between degree zero Gromov-Witten invariants and the Mac-Mahon function 
\be\label{MacM} 
\M(q):=\prod_{n=1}^\infty(1-q^n)^{-n},\quad |q|<1.
\ee
This is a classical generating function for the number of plane partitions \cite{Al},\cite[I.5.13]{Mc}. More to the point, it appears in \cite{Ktz,MNOP} in the generationg function of degree zero Donaldson-Thomas invariants of Calabi-Yau threefolds.
 
\smallskip

\noindent {\bf Donaldson-Thomas invariants.}\quad 
The Donaldson-Thomas theory provides an alternative to the Gromov-Witten description of holomorphic curves in K\"ahler manifolds, utilizing ideal sheaves instead of stable maps. Intuitively, an {\it ideal sheaf} is a collection of local holomorphic functions vanishing on a curve. This avoids counting multiple covers of the same curve separately and Donaldson-Thomas invariants are integers unlike their Gromov-Witten cousins. Counting ideal sheaves is at least  formally analogous to counting flat connections (i.e. locally constant sheaves) on a real $3$-manifold, and the Donaldson-Thomas invariants are holomorphic counterparts of the Casson invariant in Chern-Simons theory \cite{Ty}.

The genus $g$ of a stable map is replaced in a Donaldson-Thomas invariant $D_{\kappa,\alpha}$ by the {\it holomorphic Euler characteristic} $\kappa$ of an ideal sheaf. As conjectured in \cite{MNOP} and proved in \cite{Li} the degree zero partition function of a Calabi-Yau threefold $X$ is given by 
\bee
Z_X^0(q):=\sum_{\kappa=0}^\infty D_{\kappa,0}\,q^{\kappa}=\M(-q)^{\chi(X)},
\eee
where as before $\chi(X)$ is the classical Euler characteristic. 

Since both kinds of invariants are meant to describe the same geometric objects one expects a close relationship between them. Indeed, it is proved in \cite{MNOP} for toric threefolds and conjectured for general ones that reduced partition functions of Donaldson-Thomas and Gromov-Witten theories are the same under a simple change of variables. This equality does not extend directly to degree zero but it is mentioned in \cite{MNOP} that the  Gromov-Witten $F^{0}$ is the asymptotic expansion of $\ln\M(e^{ix})^{\frac12\chi(X)}$ at $x=0$ (note the extra $1/2$ in the exponent). 

A quick look at \eqref{F0rescon} tells one that even for the resolved conifold this can be true at best for $g\geq2$ since no extra variables are involved in the Donaldson-Thomas function. We will see that this is the case but the {\it complete} asymptotic expansion involves some interesting extra terms that are perplexing from the Gromov-Witten point of view. However, the Mac-Mahon factor is exactly reproduced in Chern-Simons theory (\refL{ZqBarnes}). Moreover, with asymptotic expansions one has to specify not just a point but also a {\it direction} in the complex plane in which the expansion is taken, and the correct direction here is not the obvious (real positive) one. 

\smallskip

\noindent {\bf $\zeta$-resummation.} To avoid imaginary numbers we first consider $\ln\M(e^{-x})$ instead of $\ln\M(e^{ix})$. For motivation, we start with a provocative 'computation' that converts an expansion in powers of $e^{-x}$ into one in powers of $x$ for a simpler function: 
\be\label{provoc}
\frac{e^{-x}}{1-e^{-x}}=\sum_{n=1}^\infty e^{-nx}=\sum_{n=1}^\infty\sum_{k=0}^\infty\frac{(-nx)^k}{k!}
\,{'\!\!\!='}\,\sum_{k=0}^\infty\frac{(-x)^k}{k!}\sum_{n=1}^\infty\frac1{n^{-k}}
\,{'\!\!\!='}\,\sum_{k=0}^\infty(-1)^k\frac{\zeta(-k)}{k!}x^{k}.
\ee
The last two equalities are nonsense of course: the interchange of sums is illegitimate and $\sum_{n=1}^\infty\frac1{n^{-k}}=\sum_{n=1}^\infty n^k$ is (very) divergent. It certainly does not converge to $\zeta(-k)$ for positive $k$, although by definition $\zeta(s):=\sum_{n=1}^\infty\frac1{n^s},\ \Re s>1$ is the Riemann zeta function \cite{AAR}. Nonetheless, the end result is almost correct. Indeed, by definition of Bernoulli numbers \eqref{Bern}
\be\label{negBern}
\frac{e^{-x}}{1-e^{-x}}=\frac1x\frac{x}{e^x-1}=\frac1x\sum_{j=0}^\infty B_j\frac{x^j}{j!}
\ee
\be\label{zetaval}
\text{and }\zeta(-k)=-\frac{B_{k+1}}{k+1},\ k\geq1;\quad \zeta(0)=-1/2\ \cite{AAR},\text{ so}
\ee
\be\label{asBern}
\frac{e^{-x}}{1-e^{-x}}=\frac1x-\frac12+\sum_{j=2}^\infty\frac{B_j}{j!}\,x^{j-1}
=\frac1x-\frac12-\sum_{k=1}^\infty\frac{B_{k+1}}{(k+1)!}\,x^k=\frac1x+\sum_{k=0}^\infty(-1)^k\frac{\zeta(-k)}{k!}\,x^{k}.
\ee
In other words our 'computation' \eqref{provoc} only missed the first term $1/x$.

A similar feat can be performed with $\ln\M(e^{-x})$. First, we compute
\begin{multline}\label{resum}
\ln\M(e^{-x})=-\sum_{n=1}^\infty n\ln(1-e^{-nx})=\sum_{n=1}^\infty n\sum_{k=1}^\infty\frac{(e^{-nx})^k}k
=\sum_{k=1}^\infty\frac1k\sum_{n=1}^\infty n(e^{-kx})^n\\
=\sum_{k=1}^\infty\frac1k\frac{e^{-kx}}{(1-e^{-kx})^2}=\sum_{k=1}^\infty\frac1k\frac1{(e^{\frac{kx}2}-e^{-\frac{kx}2})^2}
=\sum_{k=1}^\infty\frac{\csch^2(\frac{kx}2)}{4k}.
\end{multline}
So far all the manipulations are legitimate assuming $x>0$, although they would not be if we used $e^{ix}$ instead of $e^{-x}$. Next recall the Laurent expansion at zero of $\csch^2$:
\be
\csch^2(z)=-\sum_{g=0}^\infty\frac{2^{2g}\,(2g-1)\,B_{2g}}{(2g)!}\,z^{2g-2}.
\ee
One can now pull the same trick as in \eqref{provoc} of interchanging sums and replacing divergent power sums of integers with zeta values. Namely,
\begin{multline}\label{Maczeta}
\ln\M(e^{-x})=-\sum_{k=1}^\infty\frac1{4k}\sum_{g=0}^\infty\frac{4\,(2g-1)\,B_{2g}}{(2g)!}\ 2^{2g-2}
\left(\frac{kx}2\right)^{2g-2}\\
\ {'\!\!\!='}\ -\sum_{g=0}^\infty\frac{(2g-1)\,B_{2g}}{(2g)!}\,x^{2g-2}\sum_{k=1}^\infty\frac1{k^{3-2g}}
\ {'\!\!\!='}\ -\sum_{g=0}^\infty\frac{(2g-1)\,B_{2g}\,\zeta(3-2g)}{(2g)!}\,x^{2g-2}\\
=\frac{\zeta(3)}{x^2}-\frac{\zeta(1)}{12}+\sum_{g=2}^\infty\frac{(2g-1)\,B_{2g}\,B_{2g-2}}{(2g-2)\,(2g)!}\,x^{2g-2},
\end{multline}
where we used \eqref{zetaval} in the last equality. This series is even more problematic than the one in \eqref{provoc} which at least made sense and converged for $|x|<2\pi$. Now not only does it diverge factorially (see \eqref{estBern}) but also $\zeta(1)$ makes no sense at all since $\zeta$ has a pole at $1$. Nonetheless, dropping the singular term $\frac{\zeta(3)}{x^2}$, the 'infinite constant' $-\,\frac{\zeta(1)}{12}$ and formally replacing $x$ by $-ix$ in the sum, we get exactly the higher genus Gromov-Witten free energy in degree zero \eqref{F0rescon}.

The procedure used in \eqref{provoc}, \eqref{Maczeta} can be traced back to Euler and in a more sophisticated guise is used in quantum field theory under the name of $\zeta$-resummation or $\zeta$-regularization \cite{Ac}. The amazing fact is not that this is reasonable to do in physics (one can argue that $\zeta(-k)$ has the same operational properties as the non-existent $\sum_{n=1}^\infty n^k$), but that it actually produces nearly mathematically correct answers. Unlike a physical situation, where a sensible answer is taken as a definition for an otherwise meaningless quantity, here we have an identity where both sides make perfect sense (as a holomorphic function and its asymptotic expansion respectively) and only the passage from left to right is odious.

\smallskip

\noindent {\bf Mellin asymptotics.}\quad
A fix is a well-known Mellin transform technique that not only takes care of singular terms, divergent expansions and infinite constants but even explains why the double blunder in \eqref{provoc} and \eqref{Maczeta} computes most of the asymptotic correctly \cite{FGD}. We use it here to make the relationship between the degree zero invariants and the MacMahon function precise. Recall that given an integrable function on $(0,\infty)$ with a possible pole at $0$ and polynomial decay at $\infty$ its {\it Mellin transform} is 
$$
\gotM f(s):=\int_0^\infty\,x^{s-1}f(x)\,dx.
$$
The transform is defined and holomorphic in the {\it convergence strip} $\Re s\in(\alpha,\beta)$, when $f(x)\sim O(x^{-\alpha})$ at $0$ and $\sim O(x^{-\beta})$ at $\infty$, assuming $\alpha<\beta$. It is most useful when $\gotM f$ admits a meromorphic continuation to the entire complex plane since location of the poles determines asymptotic behavior of the function at $0$ and $\infty$ (see \cite{FGD} and below). For example, $\gotM[e^{-nx}](s):=\G(s)/n^s$ in $\Re s\in(0,\infty)$ extends meromorphically with the poles of the gamma function located at $s=0,-1,-2,\dots$. Analogously, 
\be\label{MlnBern}
\gotM\left[\frac{e^{-x}}{1-e^{-x}}\right](s)=\sum_{n=1}^\infty\gotM[e^{-nx}](s)=
\sum_{n=1}^\infty\frac{\G(s)}{n^s}=\G(s)\zeta(s)\quad\text{in }\Re s\in(1,\infty)
\ee
extends with one additional zeta pole at $s=1$. 

The {\it inverse Mellin transform} recovers $f$ as 
\be\label{invMln}
f(x)=\frac1{2\pi i}\int_{c-i\infty}^{c+i\infty}\gotM f(s)\,x^{-s}\,ds\quad\text{for }c\in(\alpha,\beta)
\ee
assuming absolute integrability along $\Re s=c$. In the cases of interest to us all the poles are located on the real axis to the left of $\alpha$. 
\begin{figure}
\centering
\includegraphics[width=2truein]{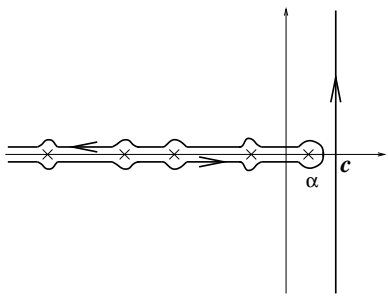} \caption{Barnes contour for Mellin transforms}\label{Barcont}
\end{figure}
If the transform satisfies appropriate growth estimates one can shift the integration contour in \eqref{invMln} to run counterclockwise along the real axis from $-\infty$ to $\alpha$ and back, \refF{Barcont}. This reduces 
\eqref{invMln} to a sum over residues at the poles by the Cauchy residue theorem:
\be\label{sumRes}
f(x)=\sum_{n=0}^\infty\ \Res_{s=-\gamma_n}\,[\gotM f(s)]\ x^{\gamma_n}.
\ee
If $\gamma_n$ are integers and the series converges $f$ must be real-analytic on $(0,\infty)$ with at worst a pole at $0$, and the residues give its Laurent coefficients at $0$. For example, one can compute the Taylor expansion of $e^{-x}$ at $0$ using that $\gotM[e^{-x}](s):=\G(s)$ and the poles $-\gamma_n=-n$ of $\G$ are simple with the residues $(-1)^n$ \cite{AAR}. However, in most cases the series \eqref{sumRes} diverges for all $x\neq0$ and \eqref{MlnBern} is such a case. Under analytic assumptions that we do not reproduce here the following weakening of \eqref{sumRes} is still true \cite{FGD}:
\begin{quote}
{\it If $-\gamma_n$ are order $m_n$ poles of (meromorphic continuation of) $\gotM f(s)$ and its Laurent expansions at $-\gamma_n$ have the form
$$
\gotM f(s)=\frac{A_{n0}}{s+\gamma_n}+\frac{A_{n1}}{(s+\gamma_n)^2}+\sum_{k=2}^{m_n-1}\frac{A_{nk}}{(s+\gamma_n)^{k+1}}
$$
then an asymptotic expansion of $f$ at $x=0$ is} 
\be\label{asMln}
f(x)\sim\sum_{n=0}^\infty\left(A_{n0}-A_{n1}\ln x+\sum_{k=2}^{m_n-1}\frac{(-1)^kA_{nk}}{k!}\ln^k x\right)x^{\gamma_n}.
\ee
\end{quote}
Now it becomes clear where the extra $1/x$ in \eqref{asBern} came from. In addition to gamma poles in \eqref{MlnBern} that produce terms $\frac{(-1)^n}{n!}\zeta(-n)x^n$ there is also a simple pole of $\zeta(s)$ with residue $1$ that gives $\G(1)\cdot1/x=1/x$. Thus, \eqref{asBern} is at least an asymptotic expansion of $e^{-x}/(1-e^{-x})$ at $0$. The fact that it actually converges to the function is a rare bonus. In general, even if \eqref{asMln} does converge it is not necessarily to the original function, see \cite{FGD}.

This technique extends to general {\it Fourier sums} (or harmonic sums) of the form 
\be\label{harsum}
f(x)=\sum_{k=1}^\infty a_k\,g(\omega_k x)
\ee
because their Mellin transforms can be easily expressed in terms of those of the {\it base} function $g$ \cite{FGD}. One can think of them as sums of generalized harmonics with {\it amplitudes} $a_k$ and {\it frequencies} $\omega_k$, the usual ones corresponding to $g(x)=e^{ix}, \omega_k=\pm k$. Indeed,
$$
\gotM f(s)=\sum_{k=1}^\infty a_k\,\int_0^\infty\,x^{s-1}g(\omega_k x)\,dx
=\sum_{k=1}^\infty \frac{a_k}{\omega_k^s}\,\int_0^\infty\,x^{s-1}g(x)\,dx=D(s)\gotM g(s),
$$
where $D(s):=\sum_{k=1}^\infty a_k/\omega_k^s$ is the {\it Dirichlet series} of the sum. If $D(s)$ is entire and $\gotM g(s)$ only has simple poles at $s=0,-1,-2,\dots$ then 
\bee
f(x)\sim\sum_{n=0}^\infty\ \Res_{s=-n}\,[\gotM g(s)]\,D(-n)\,x^n.
\eee
If moreover $g$ itself is entire and decays fast enough on $\R_+$ then $g(x)=\sum_{n=0}^\infty\,g_n\,x^n$, $\Res_{s=-n}\,[\gotM g(s)]=g_n$ and $f(x)\sim\sum_{n=0}^\infty g_n\,D(-n)\,x^n$. The same answer can be obtained by a (legitimate under the circumstances) interchange of sums in \eqref{harsum}:
$$
f(x)=\sum_{k=1}^\infty a_k\sum_{n=0}^\infty g_n(\omega_k x)^n
=\sum_{n=0}^\infty g_n\left(\sum_{k=1}^\infty a_k\,\omega_k^n\right)x^n
=\sum_{n=0}^\infty g_n\,D(-n)\,x^n
$$
In particular, this expansion is not just asymptotic but convergent. If $D(s)$ is not entire but only meromorphic the last two equalities fail. However, \eqref{asMln} still ensures that {\it formal interchange of sums gives the regular part of the asymptotic expansion correctly as long as $D$-poles are real-part positive\,!} This is precisely what happened in \eqref{provoc}.

\smallskip

\noindent {\bf Ramanujan-Wright expansion.}\quad 
The situation in \eqref{Maczeta} is more complicated. We compute from \eqref{resum}
$$
\gotM\left[\ln\M(e^{-x})\right](s)=\sum_{k=1}^\infty\frac1k\sum_{n=1}^\infty n\,\gotM[e^{-nkx}](s)
=\sum_{k=1}^\infty\frac1k\sum_{n=1}^\infty\frac{n\G(s)}{(nk)^s}
$$
Now assume that $\Re s$ is large enough for the double series to converge absolutely, e.g. $\Re s>2$, and proceed
\be\label{MacMln}
=\sum_{k,n=1}^\infty\frac{\G(s)}{n^{s-1}k^{s+1}}=
\G(s)\sum_{n=1}^\infty\frac1{n^{s-1}}\cdot\sum_{k=1}^\infty\frac1{k^{s+1}}=\G(s)\,\zeta(s-1)\,\zeta(s+1).
\ee
The extra zeta poles occur at $s-1,s+1=1$, i.e. $s=0,2$ and $s=0$ becomes a double pole. Formula \eqref{asMln} now yields an asymptotic expansion for $\ln\M(e^{-x})$ that we state as a theorem. This is a particular case of asymptotic expansions for analytic series obtained by Ramanujan who used a rough equivalent of Mellin asymptotics, the Euler-Maclaurin summation (see 
\cite[Theorem 6.12]{Brn}). Ramanujan's considerations were heuristic and in any case remained unpublished until much later. The first rigorous asymptotic for $\ln\M(e^{-x})$ is due to Wright \cite{Wr}. We sketch a proof for the convenience of the reader.
\begin{theorem}[Ramanujan-Wright]\label{MacAs} Let $\M(q):=\prod_{n=1}^\infty(1-q^n)^{-n},\ |q|<1$ be the MacMahon function. Then $\ln\M(e^{-x})$ has the Mellin transform $\gotM\left[\ln\M(e^{-x})\right](s)=\G(s)\,\zeta(s-1)\,\zeta(s+1),\ \Re s>2$ and its asymptotic expansion at $x=0$ along $\R_+$ is 
\be\label{Macas}
\ln\M(e^{-x})\sim\frac{\zeta(3)}{x^2}+\frac{\ln x}{12}+\zeta'(-1)
+\sum_{g=2}^\infty\frac{(2g-1)\,B_{2g}\,B_{2g-2}}{(2g-2)\,(2g)!}\,x^{2g-2}.
\ee
\end{theorem}
\begin{proof} Recall that $\zeta(s)$ has "trivial zeros" at negative even integers \cite{AAR}. Poles of $\G(s)$ at negative odd integers are therefore canceled by zeros of $\zeta(s-1)$. Analytical assumptions needed for \eqref{asMln} to hold are satified here by the classical estimates for $\G$ and $\zeta$ \cite{AAR}. The contributing poles are:
\begin{itemize}
\item[{(i)}] Gamma poles at $s=-2,-4,\dots,-2g,\dots$ with residues $\frac{(-1)^{2g}}{(2g)!}\zeta(-2g-1)\zeta(1-2g)$;
\item[{(ii)}] Simple pole of $\zeta(s-1)$ at $s=2$ with residue $1\cdot\G(2)\zeta(3)=\zeta(3)$;
\item[{(iii)}] Double pole of $\G(s),\zeta(s+1)$ at $s=0$.
\end{itemize}
We have from the first two items and \eqref{zetaval}
$$
\frac{\zeta(3)}{x^2}+\sum_{g=1}^\infty\frac1{(2g)!}\zeta(-2g-1)\zeta(1-2g)\,x^{2g}
=\frac{\zeta(3)}{x^2}+\sum_{g=2}^\infty\frac{(2g-1)\,B_{2g}\,B_{2g-2}}{(2g-2)\,(2g)!}\,x^{2g-2}.
$$
To take care of the double pole we need more than just the residue. By the well-known properties of $\G$ and $\zeta$
\begin{align*}
\G(1+s) &=\ \ \ 1\ \ -\ \gamma s+O(s^2)\notag\\
\zeta(s) &=\frac1{s-1}+\gamma+\,O(s-1),
\end{align*}
where $\gamma$ is the Euler constant. Thus 
$$
\G(s)\,\zeta(s+1)=\frac{\G(s+1)\,\zeta(s+1)}s=\left(\frac1{s}-\gamma+O(s)\right)\left(\frac1{s}+\gamma+O(s)\right)
=\frac1{s^2}+O(1),\ \text{and}
$$
\begin{multline*}
\G(s)\,\zeta(s-1)\,\zeta(s+1)=\left(\frac1{s^2}+O(1)\right)\left(\zeta(-1)+\zeta'(-1)s+O(s^2)\right)\\
=\frac{\zeta(-1)}{s^2}+\frac{\zeta'(-1)}{s}+O(1)=\frac{-1/12}{s^2}+\frac{\zeta'(-1)}{s}+O(1).
\end{multline*}
By \eqref{asMln} the corresponding terms in the asymptotic expansion are $\ln x/12+\zeta'(-1)$ and it remains to combine the expressions.
\end{proof}
In hindsight, it is amusing how much of \eqref{Macas} is visible in the naive expression \eqref{Maczeta}: not just the regular part but also $\zeta(3)/x^2$ and even $1/12$ in front of the logarithm. The only hidden term is $\zeta'(-1)$, sometimes called the Kinkelin constant \cite{Al}, and for this reason perhaps it is usually missing in physical papers.

\smallskip

\noindent {\bf Stokes phenomenon and the natural boundary.}\quad
As already mentioned the relationship between $q$ and $x$ is $q=e^{ix}$ not $q=e^{-x}$. Replacing formally $x$ by $-ix$ in 
\eqref{Macas} we recover the infinite sum of \eqref{F0rescon} along with three extra terms
$$
-\frac{\zeta(3)}{x^2}+\frac{\ln(-ix)}{12}+\zeta'(-1).
$$
How legitimate is this substitution? Had \eqref{Macas} been a convergent Laurent expansion there would be no such question. But it is asymptotic and represents $\ln\M(e^{-x})$ only up to exponentially small terms (more precisely, "faster than polynomially small" but we follow the standard abuse of terminology). It is well-known that such expansions depend on a direction in the complex plane in which they are taken. As one crosses certain {\it Stokes lines} originating from the center of expansion exponentially small terms may become dominant and change the expansion drastically. This change is commonly known as the {\it Stokes phenomenon}. Moreover, for an asymptotic expansion in some direction to exist the function must be holomorphic in a punctured local sector containing this direction in its interior. Switching from $x$ to $-ix$ while keeping $x$ real positive forces us to approach $q=e^{i0}=1$ along the upper arc of the unit circle, i.e. along a purely imaginary direction. For an asymptotic expansion in this direction we need to have $\M(q)$ analytically continued beyond the unit disk $|q|<1$. But can it be  continued?

The definition \eqref{MacM} does not look very promising. In fact, it strongly suggests that $\M(q)$ has a singularity at each root of unity. But roots of unity are dense on the circle making it a {\it natural boundary} for $\M(q)$ and no analytic continuation exists. It turns out to be quite hard to turn this observation into a proof but Almkvist shows \cite{Al} that if $a/b$ is a proper irreducible fraction then 
\bee
\ln\M(e^{2\pi i\frac{a}{b}}e^{-x})\ \sim\ \ \frac{\zeta(3)}{b^3\,x^2}+\frac{b}{12}\ln x+O(1)
\eee
for real positive $x$. Thus, every root of unity is indeed singular and $|q|=1$ is the natural boundary.

This forces us to reconsider keeping $x$ real in $\ln\M(e^{ix})$. Should $x$ approach $0$ from the {\it positive imaginary} direction we can set $x=iy$ with $y>0$ and \refT{MacAs} gives us an asymptotic expansion in $y$. We can rewrite it as an expansion in $x$ of course as long as it is understood that $x$ in it is positive imaginary. This may seem like an underhanded trick but it is not. The natural domain of $\ln\M(e^{ix})$ is the upper half-plane and the only distinguished direction in its interior is the positive imaginary one.
\begin{corollary}\label{MacAsIm}
Asymptotic expansion of $\ln\M(e^{ix})$ at $x=0$ along $i\,\R_+$ is (taking the principal branch of the logarithm)
\begin{multline}\label{AsIm}
\ln\M(e^{ix})\sim-\frac{\zeta(3)}{x^2}+\frac{\ln(-ix)}{12}+\zeta'(-1)
+\sum_{g=2}^\infty(-1)^{g-1}\frac{(2g-1)\,B_{2g}\,B_{2g-2}}{(2g-2)\,(2g)!}\,x^{2g-2}\\
\sim-\frac{\zeta(3)}{x^2}+\frac{\ln x}{12}+\zeta'(-1)-\frac{\pi i}{24}
+\sum_{g=2}^\infty(-1)^{g-1}\frac{(2g-1)\,B_{2g}\,B_{2g-2}}{(2g-2)\,(2g)!}\,x^{2g-2}.
\end{multline}
\end{corollary}
Comparing this to \eqref{F0rescon} one ought to be somewhat perplexed. If we are to take \eqref{AsIm} at face value then 
$p_3(t)=-\zeta(3)$, $p_1(t)=\zeta'(-1)-\pi i/24$ (?!) and there is no space for $\ln x/12$ at all. Aside from the fact that $p_i$-s are supposed to be homogeneous polynomials of the corresponding degree the numbers involved are not even rational, $\zeta(3)$ by Ap\'ery's famous result. Nevertheless, the MacMahon factor appears as is in the Chern-Simons partition function, see \refL{ZqBarnes}.

The disappearence of extra variables and appearence of irrationals suggests that some kind of averaging is involved. It would not explain $\ln x/12$ but we may guess that averaging of $p_1(t)$ is divergent and has to be regularized giving rise to an anomalous term. Why Donaldson-Thomas theory does not reproduce the degree zero contributions in low genus is beyond our expertise. However, from the Chern-Simons vantage point this ought to be expected. The idea of large $N$ duality is that the same string theory is realized on manifolds with different topology \cite{AK,Mar2}. However, the degree zero terms in genus $0,1$ are exactly the ones that record the {\it classical cohomology} of the target manifold, see \eqref{classic}. Although some relation between topologies of manifolds supporting equivalent string theories may be expected, the entire cohomology ring is certainly too much to survive a geometric transition. Therefore, these classical terms can not enter an invariant partition function except via averages that remain unchanged by such transitions.

\section{Topological vertex and\\ partition function of the resolved conifold}\label{S4}

This section and the next are to be read in conjunction. We review the salient points of two combinatorial models, the topological vertex \cite{AKMVv,Kh,L3Z,Mar2} and the Reshetikhin-Turaev calculus \cite{AK,T}, highlighting the differences but more importantly the parallels between them. The former computes the Gromov-Witten invariants of toric Calabi-Yau threefolds and the latter computes the Chern-Simons invariants of all closed $3$-manifolds. The reason to compare them is the conjectural large $N$ duality between the two. Both models encode their spaces into labeled diagrams and then assign values to them according to Feynman-like rules. However, the encoding and the rules are quite different despite intriguing correspondences. The reason we use the topological vertex instead of just summing up \eqref{resconfren} as in \cite{FP} is that it directly gives the partition function in correct variables and in an appealing form. Comparing the answer to the Chern-Simons one it becomes reasonable to express it in a closed form via the quantum Barnes function (\refT{Comparison}).
 
\smallskip

\noindent {\bf Toric webs.}\quad 
Just as the Reshetikhin-Turaev calculus \cite{AK,T} the topological vertex is a {\it diagrammatic state-sum model}. This means that geometry of a space is encoded into a diagram, a graph enhanced by additional data, and the value of an invariant is computed by summing over all prescribed labelings of the diagram. In the Reshetikhin-Turaev calculus the diagrams are link diagrams representing $3$-manifolds via surgery \cite{Rf,T}. In the topological vertex they are toric webs representing toric Calabi-Yau threefolds. 

A {\it toric web} is an embedding of a trivalent planar graph with compact and non-compact edges into $\R^2$ that satisfies some integrality conditions \cite{Kh,L3Z,Mar2}. Namely, vertices have integer coordinates and direction vectors of edges can be chosen to have integer coordinates. Moreover, if the direction vectors are chosen primitive (without a common factor in coordinates) any pair of them meeting at a vertex forms a basis of $\Z^2$ and every triple at a vertex if directed away from it, adds up to zero. Examples for the resolved conifold $\calO(-1)\oplus\calO(-1)$ and the local  $\CP^1\times\CP^1$ (i.e. the total space of $T^{1,0}(\CP^1\times\CP^1)$) are shown on \refF{torwebs}, where the primitive directions of non-compact edges are also indicated. Toric webs related by a $GL_2\Z$ transformation and an integral shift represent isomorphic threefolds. For this reason we did not label the vertices on \refF{torwebs}, one may assume that one of them is $(0,0)$ and all compact edges have the unit length. The toric web is a complete invariant of a toric Calabi-Yau. Indeed, the moment polytope of the torus action can be recovered from it \cite[4.1]{L3Z} and therefore the threefold itself up to isomorphism by the Delzant classification theorem \cite{dS}. Analogously, a $3$-manifold is recovered from its link diagram up to diffeomorphism by surgery on the link \cite{Rf,T}.
\begin{figure}
\centering
\includegraphics[width=3truein]{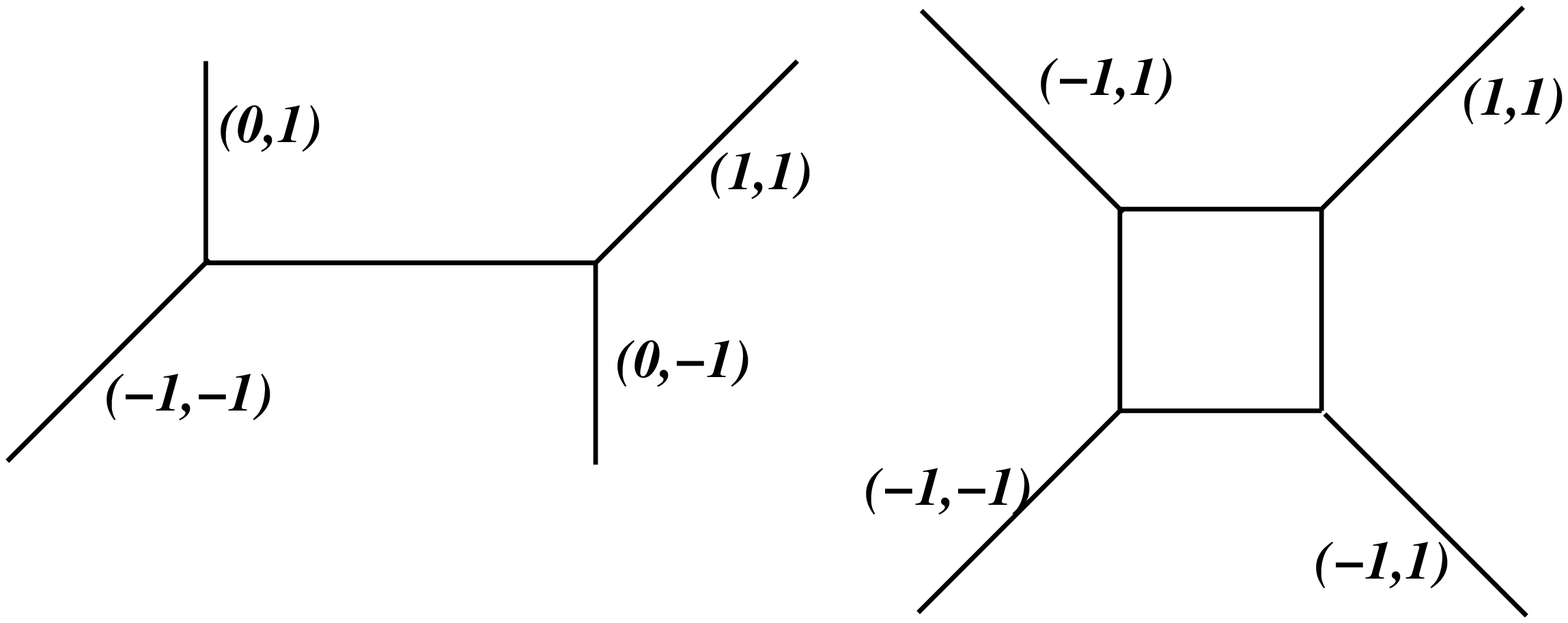} \caption{Toric webs of $\calO(-1)\oplus\calO(-1)$ and local $\CP^1\times\CP^1$}\label{torwebs}
\end{figure}
\begin{figure}
\centering
\includegraphics[width=3truein]{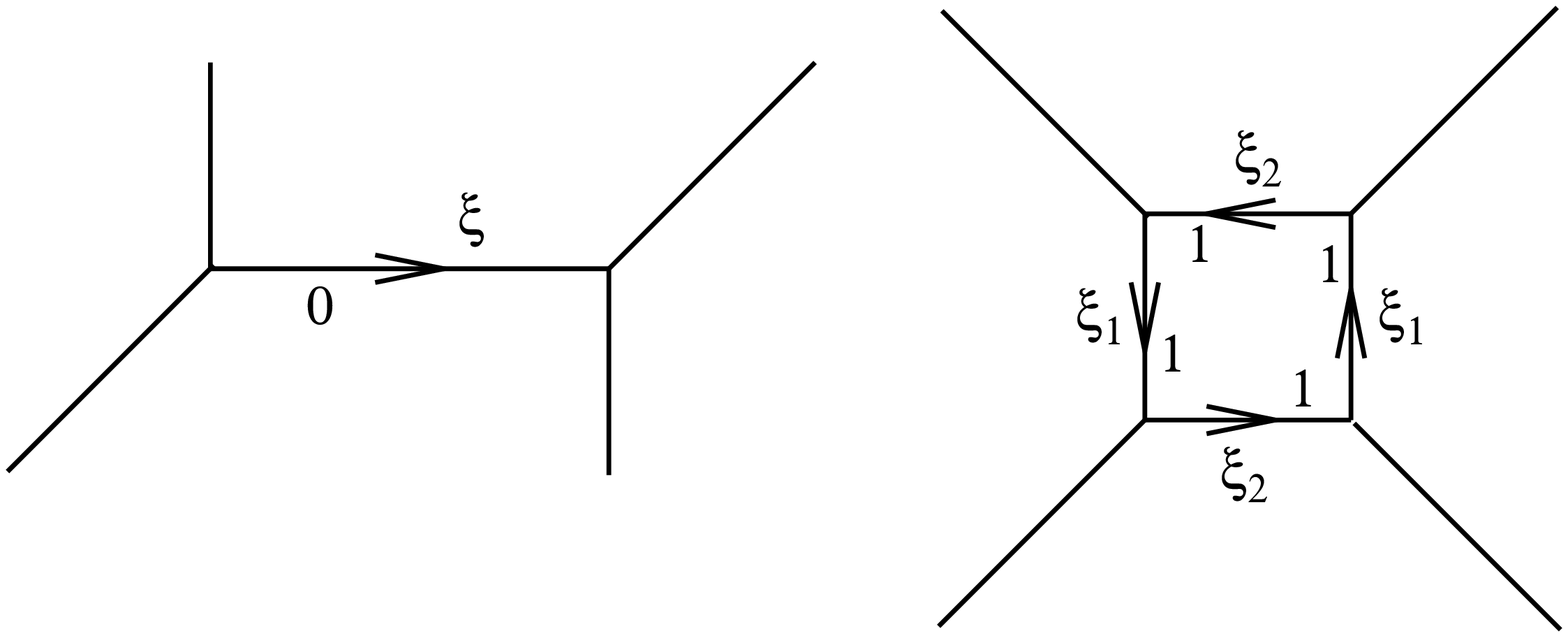} \caption{Toric graphs of $\calO(-1)\oplus\calO(-1)$ and local $\CP^1\times\CP^1$}\label{torgraphs}
\end{figure}
Having toric webs rigidly embedded in $\R^2$ is inconvenient; one would prefer to treat them as abstract graphs, perhaps with additional data. This is possible at least as far as the topological vertex is concerned although the resulting graphs may no longer be complete invariants. 

Tracing back the construction of a threefold from its web one concludes that the vertices correspond to fixed points of the torus action and compact edges correspond to fixed rational curves (copies of $\CP^1$). Being rational curves sitting inside a Calabi-Yau threefold their normal bundles are isomorphic to $\calO(n-1)\oplus\calO(-n-1),\ n=\pm1,\pm2,\dots$. The {\it framing number} $n_e$ for each edge $e$ is assigned the value from the normal bundle type of the corresponding curve. This only determines $n_e$ up to sign and the edge must be oriented to specify it. Although on their own these orientations are chosen arbitrarily they must be aligned with the framing numbers, the exact rule is given in \cite[4.2]{L3Z}. 

If $\xi_1,\dots,\xi_k$ is an integral basis in $H_2(X,\Z)$ as in \refS{S2} then each edge curve $C_e$ represents a homology class expressible as a linear combination $[C_e]=m_1\xi_1+\dots+m_k\xi_k$, $m_i\in\Z$. One requires these {\it homology relations} to be attached to the edges as well. The result is a graph called the {\it toric graph}. Toric graphs for $\calO(-1)\oplus\calO(-1)$ and the local  $\CP^1\times\CP^1$ are shown on \refF{torgraphs}. Framing numbers and homology relations are the only data aside from the topology of the web used in the topological vertex. We emphasize that both can be recovered algorithmically from the web itself without any recourse to the original threefold \cite{IqK},\,\cite[4.1]{L3Z}.

\smallskip

\noindent {\bf Partitions and Schur functions.}\quad 
We wish to briefly describe the topological vertex algorithm to see how $q$-bifactorials naturally emerge from it. This requires some basic information about partitions \cite{Mc} that appear in the Reshetikhin-Turaev calculus as well. Partitons serve as labels in state sums defining the invariants. A {\it partition} $\lambda$ is an element of $\Z_+^\infty$ with only finitely many non-zero entries that are nonincreasing, i.e.
$$
\lambda=(\lambda_1,\lambda_2,\dots,\lambda_N,0,\dots),\quad\lambda_i\in\Z_+,
\quad\lambda_1\geq\lambda_2\geq\dots\geq\lambda_N\geq0.
$$
Let $\P$ denote the set of all partitions. The number of non-zero entries $l(\lambda)$ is called the {\it length} of a partition and the sum of all entries $|\lambda|:=\lambda_1+\lambda_2+\dots+\lambda_N$ is called its {\it size} (or weight). Partitions are visualized by {\it Young diagrams}, rows of boxes stacked top down with $\lambda_i$ boxes in $i$-th row, \refF{Young}. The {\it conjugate partition} $\lambda'$ is obtained visually by transposing the Young diagram along the main diagonal and analytically as $\lambda_i':=\max\{j|\,\lambda_j\geq i\}$. Note that $\lambda''=\lambda$ and $\lambda'_1=l(\lambda)$, $|\lambda'|=|\lambda|$. Another relevant characteristic of a partition, sometimes called its {\it quadratic Casimir}, is
\be\label{kap} 
\k(\lambda):=\sum_{i=1}^\infty\lambda_i(\lambda_i-2i+1),\quad\k(\lambda')=-\k(\lambda),\quad\k(\lambda)\in2\Z.
\ee
Partitions represent possible states of compact edges in a toric graph and a combination of partition labels for each edge represents a state of the graph \cite{Kh}. The partition function is then obtained by summing over all possible states. 
\begin{figure}
\centering
\includegraphics[width=3truein]{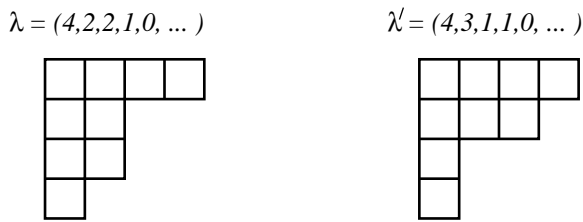} \caption{Young diagram and its conjugate}\label{Young}
\end{figure}

Amplitudes (see \refD{Amp}) of a labeled graph are defined via a specialization of {\it Schur functions} $s_\lambda$ indexed by partitions. They are symmetric 'functions' in the sense of Macdonald \cite{Mc}, i.e. formal infinite sums of monomials in countably many variables that become symmetric polynomials if all but finitely many variables are set equal to zero (more technically, if monomials containing any variable outside of a finite set are discarded from the sum). For instance, if $\lambda=(1^n):=(\underbrace{1,1,\dots,1}_{n\text{ times}},0,\dots)$ then $s_{(1^n)}$ is the $n$-th elementary symmetric function:
\be\label{elsym} 
s_{(1^n)}(x)=e_n(x):=\!\!\!\!\sum_{1\leq  i_1<\dots<i_n<\infty}\!\!\!\!x_{i_1}\cdots x_{i_n}.
\ee
In general, $s_\lambda$ are polynomials in the elementary symmetric functions given by the {\it Jacobi-Trudy formula} $s_\lambda=\det(e_{\lambda_i'-i+j}),\ 1\leq  i,j\leq\,l(\lambda')=\lambda_1$. For example,
\bee
s_{(2,1,0,\dots)}(x)=\begin{vmatrix}e_2&e_0\\e_3&e_1\end{vmatrix}=e_1e_2-e_0e_3
=\sum_{i=1}^\infty x_i\cdot\sum_{i<j=1}^\infty x_ix_j-\sum_{i<j<k=1}^\infty x_ix_jx_k.
\eee
Since $e_n$ are homogeneous of degree $n$ the Jacobi-Trudy formula implies that $s_\lambda$ are also homogeneous of degree $|\lambda|$, i.e. $s_\lambda(ax)=a^{|\lambda|}s_\lambda(x)$. Moreover, $s_\lambda,\ \lambda\in\P$ form a linear basis in the space of symmetric functions, in particular $s_\lambda s_\mu=\sum_{\nu\in\P}c^\lambda_{\mu\nu}s_\nu$. It turns out that $c^\lambda_{\mu\nu}$ are non-negative integers that vanish unless $|\nu|=|\lambda|+|\mu|$. They are the famous {\it Littlewood-Richardson coefficients} \cite{Mc}.

Specializations of Schur functions appearing in the topological vertex are obtained by specializing the formal variables $x_i$ to elements of a geometric series possibly modified at finitely many entries. Such specializations were extensively studied by Zhou \cite{Zhc}. Define the {\it Weyl vector} $\rho$ by
$$
\rho:=(-\frac12,-\frac32,\dots)=\left(-i+\frac12\right)_{i=1}^\infty.
$$
Note that $\rho$ is not a partition. Introduce a new formal variable $q$ and for any vector $\xi$ set $q^\xi:=(q^{\xi_1},q^{\xi_2},\dots)$ so in particular $q^\rho=(q^{-\frac12},q^{-\frac32},\dots)$ is a geometric series.
\begin{definition}
One-, two- and three-point functions of the topological vertex are respectively \cite{L3Z,Zhc}
\begin{align}\label{curW}
\calW_{\lambda}(q)\,\ \ &:=s_\lambda(q^\rho)=\calW_{\lambda0}(q)\notag\\
\calW_{\lambda\mu}(q)\!\,\  &:=s_\lambda(q^\rho)s_\mu(q^{\lambda+\rho})\\
\calW_{\lambda\mu\nu}(q) &:=q^{\frac{\k(\mu)+\k(\nu)}2}
\!\!\!\!\sum_{\alpha,\beta,\gamma\,\in\,\P}\!\!\!c^\lambda_{\alpha\gamma}c^{\nu'}_{\gamma\beta}   
\ \frac{\calW_{\mu'\alpha}(q)\calW_{\mu\beta'}(q)}{\calW_{\mu0}(q)}\notag
\end{align}
\end{definition}
There is a shorter expression for the three-point function via the skew Schur functions \cite{Kh,Zhc} but we do not need it here and \eqref{curW} is somewhat reminiscent of the Verlinde formula in Chern-Simons theory \cite{Wj}. We assume $q\in\C\,\backslash\,\R_-$ and $q^{\frac12}$ is then defined by the principal branch of the square root. One can see by inspection from \eqref{elsym} that $e_n(q^{\lambda+\rho})$ converges for $|q|>1$. Since the Schur 'functions' $s_\mu$ are polynomials in $e_n$ they are also well-defined as honest functions of $q$ upon specializing to $q^{\lambda+\rho}$.

To be consistent with the usual basic hypergeometric notation \cite{GaR} we wish to switch from $|q|>1$ to $|q|<1$. This can be done using a symmetry of the two-point functions \cite{Zhc}
\be\label{qinv} 
\calW_{\lambda\mu}(q)=(-1)^{|\lambda|+|\mu|}\calW_{\lambda'\mu'}(q^{-1})
=(-1)^{|\lambda|+|\mu|}s_{\lambda'}(q^{-\rho})s_{\mu'}(q^{-\lambda'-\rho}).
\ee
This identity is a curious one since the two sides never converge simultaneously (both diverge for $|q|=1$). It has the same meaning as a more familiar identity
$$
\sum_{i=1}^\infty q^i=\frac{q}{1-q}=-\frac{1}{1-q^{-1}}=-\sum_{i=0}^\infty q^{-i}=-q\sum_{i=1}^\infty q^{-i},
$$
where the two sides never converge simultaneously either. In fact, $\calW_{\lambda\mu}(q)$ are rational functions of $q^{\frac12}$ and can be analytically continued to $\C\,\backslash\,\R_-$, \eqref{qinv} expresses this analytic continuation.

The appearence of $q$-bifactorials in partition functions is due to the {\it Cauchy identity} for Schur functions \cite{Mc,Zhc}
\be\label{Cauchy}
\sum_{\lambda\in\P}s_{\lambda}(x)\,s_{\lambda'}(y)\,u^{|\lambda|}=\prod_{i,j=1}^\infty(1+u\,x_iy_j).
\ee
If $x_i=q^{i-1}$, $y_j=q^{j-1}$ the righthand side of \eqref{Cauchy} becomes
\be\label{bifac}
\prod_{i,j=1}^\infty(1+u\,q^{i-1}q^{j-1})=\prod_{i,j=0}^\infty(1+u\,q^{i+j})=(u;q)_\infty^{(2)}.
\ee
Note that although \eqref{Cauchy} is a formal identity if both sides converge as in \eqref{bifac} it holds as a function identity.

\smallskip

\noindent {\bf Partition functions as state sums.}\quad 
Let us now inspect the state sums appearing in the topologial vertex. Let $V$ and $E_c$ denote the sets of vertices and compact edges of a toric graph respectively. Choose an arbitrary orientation for each element of $E_c$, this determines the sign of the framing numbers. Assign a formal variable $a_i$ to each element of a basis $\xi_i\in H_2(X,\Z)$ and set $a_e:=a_1^{m_1}\cdots a_k^{m_k}$ for the corresponding edge curve $[C_e]=m_1\xi_1+\dots+m_k\xi_k$. Finally label all compact edges by arbitrarily chosen partitions $\lambda_e\in\P$ and non-compact ones by the trivial partition $0\in\P$. Triples of partitions $\vec{\lambda}_v:=(\lambda_1,\lambda_2,\lambda_3)$ are then assigned to each vertex according to the following rule
\begin{quote} Starting with any of the three edges and going counterclockwise around the vertex pick the edge's label if the arrow on the edge is outgoing and its conjugate if the arrow is incoming, \refF{vertpart}.
\end{quote}
\begin{figure}
\centering
\includegraphics[width=1truein]{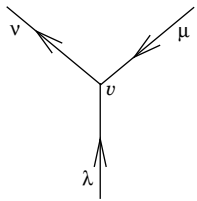} \caption{Partition triple at a vertex $\vec{\lambda}_v:=(\lambda',\mu',\nu)$}\label{vertpart}
\end{figure}
Non-compact edges present no problem since $0'=0$. This determines the triple up to cyclic permutation which is enough since $\calW_{\vec{\lambda}_v}:=\calW_{\lambda_1\lambda_2\lambda_3}$ has cyclic symmetry.
\begin{definition}\label{Amp} {\sl Amplitude} of a labeled toric graph relative to a basis $\xi_1,\dots,\xi_k\in H_2(X,\Z)$ is given by \cite{Kh,L3Z}
\be\label{amp}
A_{\{\lambda_e\}}(a_1,\dots,a_k;q):=\prod_{e\in E_c}(-1)^{|\lambda_e|(n_e+1)}\,
q^{\frac{n_e\k(\lambda_e)}2}\,a_e^{|\lambda_e|}\cdot\prod_{v\in V}\calW_{\vec{\lambda}_v}(q).
\ee
\end{definition}
The main result of \cite{L3Z} can be stated as follows
\begin{theorem}\label{Ztv} The reduced Gromov-Witten partition function of a toric Calabi-Yau threefold $X$ relative to a basis $\xi_1,\dots,\xi_k\in H_2(X,\Z)$ is given by a state sum
\be\label{ztv}
Z_X'(a_1,\dots,a_k;q)=\sum_{\{\lambda_e\}_{\lambda_e\in\P}}A_{\{\lambda_e\}}(a_1,\dots,a_k;q)
=\sum_{\lambda_1,\dots,\lambda_{|E_c|}\in\P}A_{\lambda_1,\dots,\lambda_{|E_c|}}(a_1,\dots,a_k;q)
\ee
assuming in the second sum that the edges are numbered and $\lambda_i:=\lambda_{e_i}$.
\end{theorem}

\smallskip

\noindent {\bf Partition function of the resolved conifold.}\quad 
Here $H_2(X,\Z)$ is one-dimensional and $\xi=[\CP^1]$. There is only one $a$ variable and only one edge. The amplitude $A_\lambda$ for $\lambda\in\P$ is (see \refF{torgraphs} and \eqref{qinv})
\begin{multline*}
A_{\lambda}(a;q):=(-1)^{|\lambda|}\,a^{|\lambda|}\cdot\calW_{\lambda00}(q)\,\calW_{\lambda'00}(q)
=(-a)^{|\lambda|}\calW_{\lambda}(q)\,\calW_{\lambda'}(q)\\
=(-a)^{|\lambda|}(-1)^{|\lambda|+|\lambda'|}s_{\lambda'}(q^{-\rho})s_{\lambda}(q^{-\rho})
\end{multline*}
Recalling that $|\lambda'|=|\lambda|$, $-\rho=\left(i-\frac12\right)_{i=1}^\infty$ and $s_{\lambda}$ is homogeneous of degree 
$|\lambda|$ we compute further
\bee
=(-a)^{|\lambda|}s_{\lambda'}(q^{i-\frac12})s_{\lambda}(q^{j-\frac12})
=(-a)^{|\lambda|}q^{\frac{|\lambda|+|\lambda'|}2}s_{\lambda'}((q^{i}))s_{\lambda}((q^{j}))
=(-aq^{-1})^{|\lambda|}s_{\lambda'}((q^{i}))s_{\lambda}((q^{j})).
\eee
Suppose that $a$ is small enough for $\sum_{\lambda\in\P}A_{\lambda}(a;q)$ to converge then we get by \refT{Ztv} and the Cauchy identity \eqref{Cauchy}
\begin{align}
Z_X'(a;q) &=\sum_{\lambda\in\P}\calW_{\lambda}(q)\,\calW_{\lambda'}(q)\,(-a)^{|\lambda|}
=\sum_{\lambda\in\P}s_{\lambda'}((q^{i}))s_{\lambda}((q^{j}))(-aq^{-1})^{|\lambda|}\label{z'curW}\\
&=\prod_{i,j=1}^\infty(1+(-aq^{-1})\,q^{i}q^{j})
=\prod_{i,j=0}^\infty(1-aq\,q^{i+j})=(aq;q)_\infty^{(2)}\label{z'rescon}
\end{align}
If we {\it accept} the MacMahon function as the degree zero partition function of the resolved conifold (despite the issues discussed after \refC{MacAsIm}) then
$$
Z_X^0(q)=\M(q)^{\frac{\chi(X)}2}=\frac1{\prod_{n=1}^\infty(1-q^n)^{n}}=\frac1{(q;q)_\infty^{(2)}}.
$$
We conclude that the (full) Gromov-Witten partition function of the resolved conifold is
\be\label{zrescon}
Z_X(a;q)=Z_X^0(q)\,Z_X'(a;q)=\frac{(aq;q)_\infty^{(2)}}{(q;q)_\infty^{(2)}}
\ee
as used in the Introduction.

\section{Reshetikhin-Turaev calculus and\\ partition function of the $3$-sphere}\label{S5}

As explained in the beginning of the previous section this one is complementary to it. We briefly review the $\Sl_N\C$ Reshetikhin-Turaev calculus \cite{AK,T} in a form that invites analogies with the topological vertex. In particular, we forgo the usual terminology of dominant weights and irreducible representations of $\Sl_N\C$ and rephrase everything directly in terms of partitions. The immediate goal is to compute the partition function of $S^3$ in a suitable form and compare it to the one for the resolved conifold (\refT{Comparison}).

Whereas computation of Gromov-Witten invariants in all degrees and genera is an open problem (beyond the cases of toric Calabi-Yau threefolds \cite{L3Z} and local curves \cite{BP}), Reshetikhin-Turaev calculus provides an algorithm for computing Chern-Simons invariants for arbitrary closed $3$-manifolds, especially effective for Seifert-fibered ones \cite{Mar1}. This circumstance combined with explicit large $N$ dualities for toric Calabi-Yau threefolds is the secret behind physical derivation of the topological vertex. To be sure, there is a catch. The Reshetikhin-Turaev model (or equivalently Atiyah-Turaev-Witten TQFT \cite{At,T}) is not a single model but a countable collection of them, one for each pair of positive integers $k,N$ known as {\it level} and {\it rank}. This would not be much of a hindrance if not for the tenuous connection between invariants for different $k$ and $N$. As a rule, geometers are interested in asymptotic behavior for large $k$ \cite{Roz}, and physicists in both large $k$ and large $N$ behavior. Reshetikhin-Turaev sums with ranges depending explicitly on $k$ and $N$ are not exactly custom-made for those types of questions. In fact, they require significant work even in simplest cases to be converted into asymptotic-friendly form. No general method exists; most common ad hoc procedures use Poisson resummation \cite{Mar1} or finite group characters \cite{AK,Ki2}.

The idea of the Reshetikhin-Turaev construction (related but different from Witten's original one \cite{Wj} as formalized by Atiyah \cite{At}) is to combine some deep topological results of Likorish-Wallace and Kirby with the representation theory of quantum groups \cite{AK,T}. A theorem of Likorish and Wallace asserts that any closed $3$-manifold can be obtained by  surgery on a framed link in $S^3$ \cite{Rf}. This is complemented by Kirby's characterization \cite{Kb} of links that produce diffeomorphic manifolds as those related by a sequence of Kirby moves: blow up/down and handle-slide. Blow up/down adds/removes an unknotted unlinked component with a single twist and handle-slide pulls any component over any other one, \refF{Kirby}. Thus, if one can find an invariant of framed links that remains unchanged under Kirby moves it automatically becomes an invariant of closed $3$-manifolds via surgery.
\begin{figure}
\centering
\includegraphics[width=4truein]{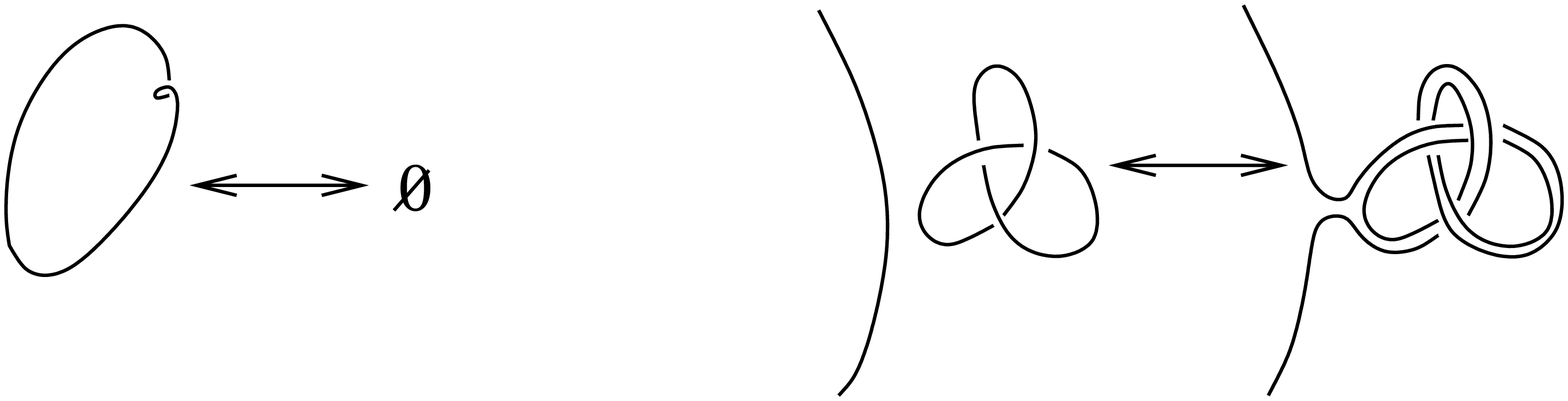} \caption{Blow up/down and handle slide over a trefoil knot}\label{Kirby}
\end{figure}

\smallskip

\noindent {\bf Hopf and twist matrices.}\quad
Framed links can be represented up to isotopy by plane diagrams with under- and over-crossings and twists as in \refF{Kirby} providing a combinatorial model of $3$-manifolds. Slicing a link diagram bottom to top and avoiding slicing through cups, caps, twists or crossings one gets arrays of basic elements \refF{linkelem} stacked on top of each other. 

This decomposition fits nicely with structure of a linear representation category: placing elements next to each other corresponds to tensoring and stacking corresponds to composition. It remains to find an object with representation category meeting all the invariance requirements. It turns out that it is extremely hard to find one producing {\it non-trivial} invariants. Classical Lie groups and algebras do not work unfortunately. One has to deform the universal enveloping algebras of say $\Sl_N\C$ into quantum groups and then specialize the deformation parameter $q$ to a root of unity $q=e^{2\pi i/(k+N)}$. As if that were not enough the tensor product of representations has to be modified as well. The end result \cite{AK,Ki2,T} is a representation-like category with only a {\it finite number of irreducible representations}. For $\Sl_N\C$ at level $k$ they are indexed by partitions with Young diagrams in the {\it $(N-1)\times k$ rectangle}, i.e.
$$
\P_{N-1}^k:=\{\lambda\in\P\left|\ l(\lambda)\leq N-1,\,l(\lambda')=\lambda_1\leq k\right.\}.
$$
In the equivalent language of dominant weights this corresponds to the weights in the Weyl alcove of the Cartan-Stiefel diagram of $\Sl_N\C$ scaled by $k$, see \cite{AK}. The Reshetikhin-Turaev invariants are computed as state sums over labelings of a link diagram with each link component labeled by a partition from $\P_{N-1}^k$, a finite set.
\begin{figure}
\centering
\includegraphics[width=3truein]{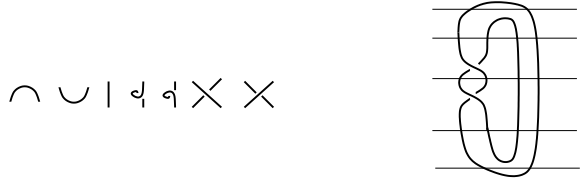} \caption{Basic elements and slicing of a Hopf link}\label{linkelem}
\end{figure}

Thus, unlike in the topological vertex where sums are taken over all partitions and are infinite, in the Reshetikhin-Turaev calculus sums are finite with explicit dependence on $k,N$. Once a diagram is labeled morphisms between irreducible representations and their tensor products are assigned to the elements from \refF{linkelem}, and then assembled by tensoring, composing and eventually taking traces (corresponding to caps) to obtain numerical invariants. The hardest ones to compute are the crossing morphisms for they depend on the so-called $R$-matrix of a quantum group \cite{AK,T}. Good news is that for a large class of $3$-manifolds, the Seifert-fibered ones and others, the use of crossing morphisms can be avoided entirely in computing the invariants \cite{Mar1,Wj} (but not in proving their invariance). In terms of conformal field theory they are completely determined by fusion rules without involving the braiding matrices \cite{Ki2}. This means that the only algebraic inputs are the {\it Hopf and twist matrices} $S$ and $T$:
\be\label{SnT}
S_{\lambda\mu}=S_{00}\,W_{\lambda\mu};\quad T_{\lambda\mu}=T_{00}\,q^{\frac12C_2(\lambda^N)}\delta_{\lambda\mu^*}\,.
\ee
The notation is as follows:
\begin{itemize}
\item[] $W_{\lambda\mu}$ is the normalized quantum invariant of the Hopf link \refF{linkelem} with components labeled by partitions $\lambda,\mu$ (see more below);
\item[] $\mu^*$ is the partition $\Sl_N$-dual to $\mu$, $\mu^*_i:=\mu_1-\mu_{N-i+1}$ for $1\leq i\leq N-1$ and $\mu^*_i:=0$\\ for $i\geq N$ (not to be confused with the conjugate partition $\mu'$);
\item[] $\lambda^N$ denotes the $\gl_N$ coordinates of $(N-1)\times k$ partition $\lambda$, $\lambda^N_i:=\lambda^N_i-|\lambda|/N$, in particular $\rho^N_i:=\frac12(N-2i+1)$;
\item[] $C_2(\lambda^N):=\lambda^N\cdot(\lambda^N+2\rho^N)$ is the quadratic Casimir of $\lambda$ as a dominant weight closely related to the quadratic Casimir $\k(\lambda)$ of a partition $\lambda$:\\ 
$\left. \qquad\qquad\qquad C_2(\lambda^N)=\k(\lambda)+N|\lambda|-|\lambda|^2/N\right.$.
\end{itemize}
We shall say more about the normalization constants $S_{00},T_{00}$ below.

Formulas for $W_{\lambda\mu}$ were originally obtained by Kac and Peterson in 1984 in the context of affine Lie algebras. Their relevance to the Chern-Simons theory was discovered by Witten \cite{Wj}. In 2001 Lukac \cite{Lc} realized that they are specializations of Schur functions of $N$ variables, namely
\be\label{WN}
W_{\lambda\mu}(N;q):=s_\lambda(q^{\rho^N})s_\mu(q^{\lambda^N+\rho^N})\ \ \text{with } q=e^{\frac{2\pi i}{k+N}}.
\ee
This should be compared to the two-point functions \eqref{curW} of the topological vertex. This realization led among other things to the physical derivation of the topological vertex, where $\calW_{\lambda\mu}(q)$ are obtained as some loosely interpreted limits of $W_{\lambda\mu}(N;q)$ \cite{AKMVv,IqK,Mar2}. Let us emphasize the differences though: in \eqref{curW} $q$ is a formal variable whereas in \eqref{WN} it is a number. Moreover, the number of variables in $s_\lambda,s_\mu$ before specialization is infinite whereas here it is $N<\infty$. This last circumstance dramatically simplifies many computations with $\calW_{\lambda\mu}$ compared to their analogs with $W_{\lambda\mu}$ because infinite specializations often have nice analytic expressions \cite{Mc}.

The twist matrix $T$ also has a vertex counterpart in the form of the framing numbers $n_e$ that contribute factors of $q^{n_e\k(\lambda_e)/2}$ to the amplitude \eqref{amp}. Incidentally, this explains their name. In the Reshetikhin-Turaev sums $T_{\lambda\mu}$ factors account for twists in link diagrams that in their turn represent {\it framing} of a link, i.e. trivialization of its normal bundle up to homotopy. If one thinks of strands as thin ribbons the signed number of twists gives exactly the signed number of full twists in a ribbon.

The two-point functions are symmetric $W_{\lambda\mu}=W_{\mu\lambda}$ and the one-point functions $W_{\lambda}:=W_{\lambda0}=W_{0\lambda}$ are called the {\it quantum dimensions} of representations indexed by $\lambda$. The {\it quantum diameter} is 
\be\label{qdiam}
\calD^2:=\sum_{\lambda\in\P_{N-1}^k}W_{\lambda}(N;q)^2
\ee
and $S_{00}$ is simply its inverse $S_{00}:=\calD^{-1}$. Analogously, $T_{00}$ is the inverse of the {\it charge factor} $\zeta:=e^{2\pi ic/24}$, where $c:=k\dim_\C(\Sl_N\C)/(k+N)$ is the so-called {\it central charge} from conformal field theory \cite{Ki2}. Thus explicitly, $T_{00}:=\zeta^{-1}=q^{-k(N^2-1)/24}$. These normalizations are needed to make the $S$ and $T$ matrices satisfy the defining relations of $SL_2\Z$
$$
(ST)^3=S^2,\quad S^2T=TS^2,\quad S^4=I.
$$
Note that unlike $W_{\lambda\mu}$ and $q^{\frac12C_2(\lambda^N)}$ that depend on $k$ only via $q=e^{2\pi i/(k+N)}$ this is no longer the case for the normalizing constants $S_{00},T_{00}$ and this causes major problems in relating Chern-Simons expressions to Gromov-Witten ones \cite{AK}.

\smallskip

\noindent {\bf Reshetikhin-Turaev invariants.}\quad
Let $J_{\lambda_1,\dots,\lambda_{n}}(L)$ denote the {\it colored HOMFLY polynomial} of an $n$-component link $L$, i.e. the amplitude of the link diagram computed as outlined above after labeling the link components by partitions ({\it colors})
$\lambda_1,\dots,\lambda_{n}$\,, see \cite{AK,T} for specifics. Then the {\it Reshetikhin-Turaev invariant} of the $3$-manifold $M$ surgered on $L$ from $S^3$ is 
\be\label{RTinv}
\tau(M):=\zeta^{-3\sigma}\,\,\calD^{-n-1}\!\!\!\!\!\!\!\!\!\sum_{\lambda_1,\dots,\lambda_{n}\in\P_{N-1}^k}
\!\!\!\!\!\!\!J_{\lambda_1,\dots,\lambda_{n}}(L)\,W_{\lambda_1}\cdots W_{\lambda_{n}}
\ee
where $\sigma$ is the signature of the linking matrix of $L$ \cite{Rf}. Since $S^3$ can be obtained from itself by surgery on the empty link $\emptyset$ with $J(\emptyset)=1$, $n=0$ components and $0$ linking matrix we have
\be\label{RTS3}
Z_{S^3}:=\tau(S^3)=\calD^{-1}=S_{00}
=\left(\sum_{\lambda\in\P_{N-1}^k}W_{\lambda}(N;q)^2\right)^{-\frac12}
\ee
This is the {\it Chern-Simons partition function of the $3$-sphere} to be identified after Witten \cite{Wc} and Gopakumar-Vafa \cite{GV} with the string partition function of $T^*S^3$. Note that \eqref{qdiam} bears some resemblance to the expression \eqref{z'curW} for the reduced partition function of the resolved conifold especially if we interpret $(-a)^{|\lambda|}$ as a convergence factor needed to extend the sum from $\P_{N-1}^k$ to all partitions. However, unlike  \eqref{z'curW} that gives $Z_{X}'$ directly  formula \eqref{qdiam} gives $Z_{S^3}^{-2}$, so naively one does not expect  these partition functions to be nearly equal. 

The normalization adopted in \eqref{RTinv} is the one that makes $\tau(M)$ into an honest invariant. It is due to Reshetikhin and Turaev and is known to differ from the physical normalization of Witten \cite{Wj} coming from the path integral. Although the physical normalization gives better agreement with Gromov-Witten theory \cite{OV2} it has not been consistently defined in general. In the known examples it is discerned by comparing $\tau(M)$ to heuristic path integral expansions \cite{Roz}.

\smallskip

\noindent {\bf Partition function of $S^3$.}\quad Although we used $q$ above as much as possible one should not forget that in the context of Chern-Simons theory this is simply a shorthand for $e^{2\pi i/(k+N)}$ and all the relevant quantities are only defined for positive integers $k,N$. Large $N$ duality predicts among other things that it should be possible to interpret $\tau(M)$ as restrictions of holomorphic in $q$ functions to these special values, perhaps up to some {\it parasite factors} steming form misnormalizations. It is these holomorphic functions that should correspond to the Gromov-Witten generating functions in $q$. According to this philosophy we should try to transform the righthand side of \eqref{qdiam} into an explicit function of $q$ and $N$ so far as possible. A clean way of doing this is not known to the author, the problem being the rogue range of summation $\P_{N-1}^k$. The known ways via Poisson resummation \cite[4.2]{Mar1} or group characters \cite[Theorem 10]{AK} are rather down and dirty and we do not reproduce them here. An intermediate answer is 
\be\label{intermid}
\calD^2=(-1)^{\frac{N(N-1)}2}\ N(k+N)^{N-1}\!\!\!\prod_{1\leq i<j\leq N} 
\!\!\left(q^{\frac12(j-i)}-q^{-\frac12(j-i)}\right)^{-2}
\ee
and it fulfils our wish only partially. One could replace $k+N$ by $\ln q/2\pi i$ but this does not lead to anything useful. 

In fact, $N(k+N)^{N-1}$ is the volume of the fundamental parallelepiped of a rescaled root lattice $\Lambda_r$ of $\Sl_N\C$, namely of $(k+N)^{1/2}\Lambda_r$. Ooguri and Vafa \cite[2.1]{OV2} replace it with the volume of $U_N$ in their normalization.
Thus, we ignore $N(k+N)^{N-1}$ as a parasite factor and focus on the product following it instead. Since
$\sum_{1\leq i<j\leq N}\!1=N(N-1)/2$ and $\sum_{1\leq i<j\leq N}(j-i)=N(N^2-1)/12$ we have
$$
\prod_{1\leq i<j\leq N}\!\!\left(q^{\frac12(j-i)}-q^{-\frac12(j-i)}\right)
=q^{-\frac{2N(N^2-1)}{24}}\prod_{1\leq i<j\leq N}\!\!\left(1-q^{j-i}\right).
$$
Note that $q^{-\frac{2N(N^2-1)}{24}}=\zeta^{-2N/k}$ and a power of $\zeta$ appears explicitly in the Reshetikhin-Turaev normalization \eqref{RTinv}. We get for the partition function
\be\label{ZS3}
Z_{S^3}:=\tau(S^3)=i^{\frac{N(N-1)}2}q^{-\frac{N(N^2-1)}{12}}\left(N(k+N)^{N-1}\right)^{-\frac12}
\prod_{1\leq i<j\leq N}\!\!\left(1-q^{j-i}\right)
\ee
The first two factors are $1$ in absolute value and can be regarded as framing corrections and we already discussed the third (volume) factor. In any case, it turns out that only the last product is relevant in the context of large $N$ duality as we now demonstrate (cf.\cite[Appendix]{Kan})
\begin{lemma}\label{ZqBarnes}For any $|q|<1$
\be\label{zqBarnes} 
\prod_{1\leq i<j\leq N}\!\!\left(1-q^{j-i}\right)=\prod_{n=1}^{N-1}\left(1-q^n\right)^{N-n}
=(q;q)_\infty^N\ \,\frac{(q^{N+1};q)_\infty^{(2)}}{(q;q)_\infty^{(2)}},
\ee
where $(a;q)_\infty:=(a;q)_\infty^{(1)}$ is the usual $q$-shifted factorial \cite{GaR} and $(a;q)_\infty^{(d)}$ was defined in the Introduction \eqref{qfac}.
\end{lemma}
\begin{proof}
First we arrange the factors according to the powers of $q$
$$
\prod_{1\leq i<j\leq N}\!\!\left(1-q^{j-i}\right)=\prod_{n=1}^{N-1}\left(1-q^n\right)^{\sum_{i=1}^{N-1}1_{\{i|n+i\leq N\}}}
=\prod_{n=1}^{N-1}\left(1-q^n\right)^{N-n},
$$
where $1_S$ denotes the characteristic function of a set $S$. A finite product can be written as a ratio of two infinite ones as long as the latter converge, in particular
\begin{multline*}
\prod_{n=1}^{N-1}\left(1-q^n\right)^{N-n}
=\frac{\prod_{n=1}^\infty\left(1-q^n\right)^{N-n}}{\prod_{n=N}^\infty\left(1-q^n\right)^{N-n}}
=\frac{\prod_{n=1}^\infty\left(1-q^n\right)^{N-n}}{\prod_{n=0}^\infty\left(1-q^{N+n}\right)^{-n}}\\
=\prod_{n=1}^\infty\left(1-q^n\right)^{N}
\ \frac{\prod_{n=1}^\infty\left(1-q^{N+n}\right)^{n}}{\prod_{n=N}^\infty\left(1-q^n\right)^{n}}
=(q;q)_\infty^N\ \,\frac{(q^{N+1};q)_\infty^{(2)}}{(q;q)_\infty^{(2)}}.
\end{multline*}
\end{proof}
Setting $a=q^N$ in \eqref{zqBarnes} and comparing to the partition function of the resolved conifold \eqref{zrescon} we see that Chern-Simons theory reproduces both the reduced partition function and the Mac-Mahon factor. It also produces an additional one $(q;q)_\infty^N$ aside from the parasite factors discussed above. This extra factor appears naturally though  in the {\it quantum Barnes function} $G_q$. Its most characteristic property is the functional equation $G_q(z+1)=\G_q(z)G_q(z)$, where $\G_q$ is the quantum gamma function of Jackson \cite{AAR}. $\G_q$ is a $q$-deformation of the classical Euler's $\G$ and the classical Barnes function satisfies the same equation with $q$-s removed \cite{Ni1}. As we derive in the next section (\refT{Altz}) to have this property $G_q$ must be the product
$$
G_q(z+1)=(1-q)^{-\frac{z(z-1)}2}\,(q;q)_\infty^z\ \frac{(q^{z+1};q)_\infty^{(2)}}{(q;q)_\infty^{(2)}},
$$
i.e. reproduce the righthand side of \eqref{zqBarnes} with $z=N$ up to a power of $1-q$. Emergence of the same expressions from a simple functional equation is quite intriguing as well as the fact that they are the ones common in Chern-Simons and Gromov-Witten theories. We summarize the observations made in the last two sections as a theorem.
\begin{theorem}\label{Comparison} Quantum Barnes function $G_q$ is the common factor of the partition functions of the resolved conifold $X$ and the $3$-sphere, namely
\begin{align}
(1-q)^{\frac{z(z-1)}2}\,Z_X(q^z;q) &=(q;q)_\infty^{-z}\ \,G_q(z+1);\label{GqZX}\\
(1-q)^{\frac{z(z-1)}2}\,Z_{S^3}(z;q) &=i^{\frac{z(z-1)}2}q^{-\frac{z(z^2-1)}{12}}z^{-\frac12}
\left(\frac{\ln q}{2\pi i}\right)^{-\frac{z-1}2}\!G_q(z+1).\label{GqZS3}
\end{align}
In \eqref{GqZX} $q$ is the usual variable of the topological vertex and $a:=q^z$ is the degree variable \cite{L3Z}. In \eqref{GqZS3} $q=e^{2\pi i/(k+N)}$ and $z=N$, where $k,N$ are level and rank. Equivalently, if $t$ is the K\"ahler parameter and $x$ the string coupling constant of \cite{Mar2} then $q=e^{ix},\,q^z=e^{-t}$.
\end{theorem}

\section{Quantum multigamma hierarchy}\label{S6}

In this section we give a self-contained introduction into the theory of the quantum Barnes function and its higher order analogs, quantum multigammas. The main result is the alternating formula \eqref{altz} for $G^{(d)}_q$ that displays its graded product structure. Along the way we describe connections to other special functions that came into the spotlight lately, $q$-shifted multifactorials and quantum polylogarithms. 

\smallskip

\noindent {\bf Quantum multigammas.}\quad Quantum multigammas emerge naturally if one iterates two classical constructions. One is the construction of factorials from natural numbers with subsequent analytic continuation to complex values via the functional equation $\G(z+1)=z\G(z)$. The other is a $q$-deformation of $\G$ first performed by Jackson, thorougly forgotten and then revived by Askey in 1970-s, see \cite{AAR}. Euler's construction was iterated by Kinkelin in 1860-s who turned $\G(z)$ into a new $z$, i.e  considered the equation $G(z+1)=\G(z)G(z)$ but only for positive integers $z$. Barnes in 1900-s introduced $G(z)$ from entirely different considerations that define it for complex values directly and lead to a hierarchy of functions with $G^{(0)}(z)=z,G^{(1)}(z)=\G(z),G^{(2)}(z)=G(z),\dots; G^{(d)}(z+1)=G^{(d-1)}(z)G^{(d)}(z)$. Nowadays $G(z)$ is known as the (classical) {\it Barnes function} and $G^{(d)}$ as {\it multigamma functions} \cite{Ni3}.

It is clear that the hierarchy of functional equations together with a normalization $G^{(d)}(1)=1$ uniquely define $G^{(d)}$ on all natural numbers. Extension to complex values can not be unique even for $\G$ because there are entire functions that vanish on all integers, $\sin(\pi z)$ for example. However, Bohr and Mollerup proved in 1920-s that if a log-concavity condition is added $\frac{d^2}{dx^2}\G(x)\geq0$ for $x\geq0$ then the Euler's $\G$ is the only possibility \cite{AAR}. In 1963 Dufresnoy and Pisot generalized the Bohr-Mollerup existence and uniqueness result to a wide class of functional equations of the type $f(z+1)-f(z)=\phi(z)$. It constructs not only $G^{(d)}$ but also the Jackson's deformation of Euler's $\G$ that satifies $\G_q(z+1)=(z)_q\,\G_q(z)$ with $(z)_q:=\frac{1-q^z}{1-q}$ called the {\it quantum number} \cite{GaR,Mc}. For higher order multigammas the log-concavity condition has to be replaced with log-positivity of order $d+1$, i.e.
$\frac{d^{d+1}}{dx^{d+1}}\ln\G(x)\geq0$ for $x\geq0$. After $q$-deformations of Barnes' multigammas appeared in the context of integrable hierarchies \cite{FZ} Nishizawa realized that the same iteration works for them as well \cite{Ni1}. In other words, there exists a unique hierarchy of meromorphic functions satisfying
\begin{align}\label{gqdef}
\text{(i)}\   &G_q^{(d)}(1)=1,\quad G_q^{(0)}(z)=\frac{1-q^z}{1-q}\notag\\
\text{(ii)}\  &G_q^{(d)}(z+1)=G_q^{(d-1)}(z)\,G_q^{(d)}(z)\\
\text{(iii)}\ &\frac{d^{d+1}}{dx^{d+1}}\,\ln G_q(x)\geq0\ \ \text{for}\ \ x\geq0,\,0<q<1\notag
\end{align}
The last condition is only required to hold for real positive $q$ but it is understood that $G_q^{(d)}$ are continued analytically to other values of $q$.
\begin{definition}[Quantum multigammas]\label{Gqdef} The functions $G_q^{(d)}$ defined by \eqref{gqdef} are called $q$-multigamma functions, in particular $\G_q:=G_q^{(1)}$ is the Jackson's quantum gamma function and $G_q:=G_q^{(2)}$ is the {\sl quantum Barnes (or $q$-Barnes) function}.
\end{definition}
\noindent For $|q|<1$ we will show the existence of $G_q^{(d)}$ independently by deriving explicit expressions for it. In addition to integrable hierarchies these functions also appear in analytic number theory, see \cite[Remark 3]{Wk}.
\smallskip

\noindent {\bf $q$-Multifactorials.} In \cite{Ni1} Nishizawa derives an explicit combinatorial formula for $G_q^{(d)}$ (see \eqref{niprod}). A more illuminating formula for our purposes can be expressed in terms of $q$-shifted multifactorials, $q$-multifactorials for short (\refT{Altz}). We streamline Nishizawa's approach by introducing them and systematically using generating functions. 

The classical $q$-shifted factorial as defined by Euler is $(a;q)_\infty:=\prod_{i=0}^\infty(1-aq^i)$. This notation is widely used in the theories of basic hypergeometric series \cite{GaR}, modular forms and partitions \cite{Mc}.
A natural generalization essentially due to Appell is
\begin{definition}[$q$-multifactorials]\label{Qmulfac}A $q$-shifted $d$-factorial is 
\begin{align}\label{qmulfac}
(a;q)_\infty^{(0)} &:=1-a\quad\text{{\rm (no $q$ dependence);}}\notag\\ 
(a;q)_\infty^{(d)} &:=\prod_{i_1,\dots,i_d=0}^\infty(1-aq^{i_1+\cdots+i_d}),\quad|q|<1,\,d=1,2,\dots\,.
\end{align}
For finite $N$ we set 
\be\label{qNmulfac}
(a;q)_N^{(d)}:=\frac{(a;q)_\infty^{(d)}}{(aq^N;q)_\infty^{(d)}}.
\ee
\end{definition}
Our indexing convention is in line with \cite{Wk} but differs from \cite{Ni2}, our $(a;q)_\infty^{(d)}$ is Nishizawa's $(a;q)_\infty^{(d-1)}$. Also one should not confuse our $q$-multifactorials with {\it multiple $q$-factorials} that have a separate variable $q_{k}$ for each index. Our definition is recovered if all $q_{k}$ are set equal to $q$. We already used $(a;q)_\infty^{(2)}$ to write partition functions in a closed form. 

Unlike the case $d=1$ the finite version $(a;q)_N^{(d)}$ is no longer a finite product of $(1-aq^i)$ unless $d$ divides $N$. Nevertheless we have 
\begin{multline*}
\frac{\prod_{i_1,\dots,i_d=0}^\infty(1-aq^{i_1+\cdots+i_d})}{\prod_{i_1,\dots,i_d=0}^\infty(1-aq^{N+i_1+\cdots+i_d})}
=\frac{\prod_{i_{d}=0}^\infty\prod_{i_1,\dots,i_{d-1}=0}^\infty(1-aq^{i_1+\cdots+i_d})}
{\prod_{i_{d}=N}^\infty\prod_{i_1,\dots,i_{d-1}=0}^\infty(1-aq^{N+i_1+\cdots+i_d})}\\
=\frac{\prod_{i=0}^\infty(aq^i;q)_\infty^{(d-1)}}{\prod_{i=N}^\infty(aq^i;q)_\infty^{(d-1)}}
=\prod_{i=0}^{N-1}(aq^i;q)_\infty^{(d-1)},
\end{multline*}
and by definition 
\be\label{qNprod}
(a;q)_N^{(d)}=\prod_{i=0}^{N-1}(aq^i;q)_\infty^{(d-1)}, \quad d\geq1.
\ee
A similar calculation gives another useful identity
\be\label{aqd}
(aq;q)_\infty^{(d)}=\frac{(a;q)_\infty^{(d)}}{(a;q)_\infty^{(d-1)}}.
\ee
Also $q$-multifactorials can be expressed as a single product if we group the factors by powers of $q$. To this end it is convenient to introduce the {\it Stirling polynomials}
\be\label{Stirpol}
\binom{z}{0}:=1\qquad\binom{z}{n}:=\frac{z(z-1)\cdots(z-n+1)}{n!},\quad n\geq1,
\ee
that reduce to the binomial coefficients when $z=N\geq n$ is a positive integer. Some of their properties are reviewed in the Appendix. From a generating function for Stirling polynomials \eqref{Stir-gen}
\begin{multline}\label{Stircount}
\sum_{i_1,\dots,i_d=0}^\infty
t^{i_1+\cdots+i_d}=\sum_{n=0}^\infty\left(\sum_{i_1+\cdots+i_d=n}\!\!\!\!1\right)\,t^n
=\left(\sum_{i=0}^\infty\,t^i\right)^d\\
=(1-t)^{-d}=\sum_{n=0}^\infty\binom{d+n-1}{n}\,t^n.
\end{multline}
In other words, $\binom{d+n-1}{n}$ is the number of arrangements of $d$ non-negative integers adding up to $n$. 
As a consequence we have {\it single product representations}
\begin{align}\label{singprod}
(a;q)_\infty^{(d)} &:=\prod_{i_1,\dots,i_d=0}^\infty(1-aq^{i_1+\cdots+i_d})
=\prod_{n=0}^\infty(1-aq^n)^{\binom{d+n-1}{n}}\\ 
(aq;q)_\infty^{(d)} &=\prod_{n=1}^\infty(1-aq^n)^{\binom{d+n-2}{n-1}}.\notag
\end{align}
In particular, $1/(q;q)_\infty^{(2)}=\prod_{n=1}^\infty(1-q^n)^{-\binom{n}{1}}=\M(q)$ as we claimed before. They also imply that $-\ln(a;q)_\infty^{(d)}$ is nothing other than the {\it quantum polylogarithm} of Kirillov \cite[Ex.\,2.5.8]{Ki1},\,\cite{Ni2} defined for $|a|,|q|<1$ as 
\be\label{qLi}
\Li_s(a;q):=\sum_{k=1}^\infty\frac{a^k}{k(1-q^k)^{s-1}}.
\ee
Indeeed, by \eqref{Stir-gen} 
\begin{multline}\label{mfacqLi}
-\ln(a;q)_\infty^{(d)}=-\sum_{n=0}^\infty\binom{d+n-1}{n}\,\ln(1-aq^n)
=\sum_{n=0}^\infty\binom{d+n-1}{n}\sum_{k=1}^\infty\frac{(aq^n)^k}{k}\\
=\sum_{k=1}^\infty\frac{a^k}{k}\sum_{n=0}^\infty\binom{d+n-1}{n}(q^k)^n
=\sum_{k=1}^\infty\frac{a^k}{k}(1-q^k)^{-d}=\Li_{d+1}(a;q)
\end{multline}
The identity holds for any $a\in\C$ by analytic continuation.
\smallskip

\noindent {\bf Closed formulas for $G_q^{(d)}$.} Typically when a classical object is $q$-deformed its theory becomes more complicated. Refreshingly, $q$-multigammas for $|q|<1$ are an exception: their theory is much simpler than its classical counterpart. The underlying reason is that it is possible to write finite products as ratios of infinite ones \eqref{qNmulfac}, the latter having a straightforward extension from integers $N$ to complex values $z$. An analogous attempt to write $N!=\frac{1\cdot2\cdot3\cdots}{(N+1)\cdot(N+2)\cdot(N+3)\cdots}$ leads to a nonsensical product of all natural numbers. The closest classical imitation is the Gauss product formula
$$
\G(z+1)=\lim_{n\to\infty}\frac{(n+1)!}{(z+1)\cdots(z+n+1)}\,n^z
$$
that requires the rather involved theory of Weierstrass products. One can fruitfully turn things around and derive product formulas for $\G$ and classical higher multigammas $G^{(d)}$ from the $q$-deformed ones via a limiting procedure \cite{Ni3}.

We leave as an exercise to the reader to iterate the functional equation for $G_q^{(d)}$ and derive by induction from \eqref{qNprod} that for non-negative integers $N$
\begin{align}\label{altN}
G_q^{(d)}(N+1) &=(q;q)_\infty^{(0)\ -\binom{N}{d}}\ (q;q)_\infty^{(1)\ \binom{N}{d-1}}
\cdots\ (q;q)_\infty^{(d)\ (-1)^{d+1}\binom{N}{0}}\ \ \ (q^{N+1};q)_\infty^{(d)\ (-1)^{d}}\notag\\
&=\prod_{i=0}^d\,(q;q)_\infty^{(i)\ (-1)^{i+1}\binom{N}{d-i}}\ \ \ (q^{N+1};q)_\infty^{(d)\ (-1)^{d}}.
\end{align}
Now $N$ can be painlessly replaced with any complex $z$ as long as we stipulate that $q\in\D\backslash\R_-$. This way first the infinite products converge since $q$ is inside the unit disk $\D$, and second $q^{z}:=e^{z\ln q}$ is defined by choosing the principal branch of the logarithm. Of course, a priori there is no guarantee that the function so extended coincides with $G_q^{(d)}$ of \refD{Gqdef}. We now prove that this is the case.
\begin{theorem}[Alternating formula]\label{Altz} Let $q\in\D\backslash\R_-$ with $\D$ the open unit disk in $\C$. 
Then $G_q^{(d)}(z)$ is an {\sl entire} function of $z$ given by 
\be\label{altz}
G_q^{(d)}(z+1)=\prod_{i=0}^d\,(q;q)_\infty^{(i)\ (-1)^{i+1}\binom{z}{d-i}}\ \ (q^{z+1};q)_\infty^{(d)\ (-1)^{d}},
\ee
where $(q;q)_\infty^{(i)}$ are the $q$-multifactorials \eqref{qmulfac} and  $\binom{z}{n}$ are the Stirling polynomials \eqref{Stirpol}. In particular, for quantum gamma and Barnes functions ($d=1,2$)
\begin{align}
\G_q(z+1) &=\frac1{(q;q)_\infty^{(0)\,\,z}}\ \ (q;q)_\infty^{(1)}\ \ \frac1{(q^{z+1};q)_\infty^{(1)}}\label{qGalt}\\
G_q(z+1) &=\frac1{(q;q)_\infty^{(0)\ \frac{z(z-1)}2}}\ \ (q;q)_\infty^{(1)\,\,z}\ \ \frac1{(q;q)_\infty^{(2)}}
\ \ (q^{z+1};q)_\infty^{(2)}.\label{qBalt}
\end{align}
Recall that $(q;q)_\infty^{(0)}:=1-q,\ (q;q)_\infty^{(1)}:=(q;q)_\infty$.
\end{theorem}
\begin{proof}
We have to verify that the righthand side of \eqref{altz} satisfies all three conditions of \eqref{gqdef}. For the duration of the proof we use $G_q^{(d)}$ only as an alias for this righthand side. Normalizations (i) follow by direct substitution of values. For the functional equation (ii) we compute
\begin{align}\label{funprod}
G_q^{(d-1)}(z)\,G_q^{(d)}(z) &=\prod_{i=0}^{d-1}\,(q;q)_\infty^{(i)\ (-1)^{i+1}\binom{z-1}{d-i-1}}\ 
\prod_{i=0}^d\,(q;q)_\infty^{(i)\ (-1)^{i+1}\binom{z-1}{d-i}}\notag\\
&\hspace{3cm} \cdot\ (q^{z};q)_\infty^{(d-1)\ (-1)^{d-1}}\ (q^{z};q)_\infty^{(d)\ (-1)^{d}}\\
&=\prod_{i=0}^{d-1}\,(q;q)_\infty^{(i)\ (-1)^{i+1}\left[\binom{z-1}{d-i-1}+\binom{z-1}{d-i}\right]}
\ \,(q;q)_\infty^{(d)\ (-1)^{d+1}}\left(\frac{(q^{z};q)_\infty^{(d)}}{(q^{z};q)_\infty^{(d-1)}}\right)^{(-1)^{d}}.\notag
\end{align}
By \eqref{Stirid} with $n=d-i$ the sum in brackets is just $\binom{z}{d-i}$ so the first two factors combine into 
$\prod_{i=0}^{d}\,(q;q)_\infty^{(i)\ (-1)^{i+1}\binom{z}{d-i}}$\!\!. For the last factor, \eqref{aqd} with $a=q^z$ yields
$$
\frac{(q^z;q)_\infty^{(d)}}{(q^z;q)_\infty^{(d-1)}}=(q^{z+1};q)_\infty^{(d)}
$$
so \eqref{funprod} becomes
$$
G_q^{(d-1)}(z)\,G_q^{(d)}(z)=\prod_{i=0}^{d}\,(q;q)_\infty^{(i)\ (-1)^{i+1}\binom{z}{d-i}}
(q^{z+1};q)_\infty^{(d)}=G_q^{(d)}(z+1).
$$
It remains to verify the log-positivity condition (iii). Taking the logarithm we have
\be\label{altzln}
\ln G_q^{(d)}(z+1)=\sum_{i=0}^d\,(-1)^{i+1}\,\ln(q;q)_\infty^{(i)}\,\binom{z}{d-i}+(-1)^{d}\,\ln(q^{z+1};q)_\infty^{(d)}.
\ee
Since $\binom{z}{d-i}$ is a polynomial in $z$ of degree $d-i$ the sum above is a polynomial of degree $d$ and its ${d+1}$-st derivative vanishes. Recalling \eqref{mfacqLi} one obtains
$$
\frac{d^{d+1}}{dz^{d+1}}\,\ln G_q^{(d)}(z)=(-1)^{d+1}\frac{d^{d+1}}{dz^{d+1}}\,\Li_{d+1}(q^z;q).
$$
It is helpful to notice that $\dfrac{d}{dz}f(q^z)=\ln q\ x\left. \dfrac{df}{dx}\right|_{x=q^z}$ and iterating
$$
(-1)^{d+1}\,\frac{d^{d+1}}{dz^{d+1}}\,f(q^z)
=\left.(-\ln q)^{d+1}\,\left(x\frac{d}{dx}\right)^{d+1}\!\!\!f(x)\right|_{x=q^z}.
$$
Now for $|x|<1$ by \eqref{qLi} 
$$
\left(x\frac{d}{dx}\right)^{d+1}\!\!\!\Li_{d+1}(x;q)
=\sum_{n=1}^\infty\frac{1}{n(1-q^n)^{d}}\left(x\frac{d}{dx}\right)^{d+1}\!\!\!x^n
=\sum_{n=1}^\infty\frac{n^{d+1}}{n(1-q^n)^{d}}\ x^n
$$
and therefore
$$
\frac{d^{d+1}}{dz^{d+1}}\,\ln G_q^{(d)}(z)
=(-\ln q)^{d+1}\sum_{n=1}^\infty\left(\frac{n}{1-q^n}\right)^{d}\!\!q^{nz}\quad\text{for }|q^z|<1.
$$
If $0<q<1$ and $z$ is real positive then $-\ln q>0$ and $0<q^z<1$ so the last expression is positive by inspection.
\end{proof}
The structure of $G_q^{(d)}$ is perhaps most transparent in \eqref{altzln}. We have the {\it main term} $(-1)^{d}\,\ln(q^{z+1};q)_\infty^{(d)}$ that depends on $a=q^z$ only and the {\it anomaly term} polynomial in $z$ of degree $d$. For $d=2$ we recognize $\ln(q^{z+1};q)_\infty^{(2)}$ as exactly the reduced free energy of the resolved conifold, see  
\eqref{z'rescon}. The main/anomaly structure of $G_q^{(d)}$ becomes even more transparent if we introduce a generating function for the entire hierarchy
\bee
\sum_{d=0}^\infty\ln G_q^{(d)}(z+1)\,t^d
=(1+t)^z\sum_{d=0}^\infty\,\,(-1)^d\Li_{d+1}(q;q)\,t^d-\sum_{d=0}^\infty\,(-1)^d\Li_{d+1}(q^{z+1};q)\,t^d.
\eee
One can check that the sum converges for $|q|,|q^{z}|<1$, $|t|<1-|q|$. This formula displays very clearly the alternating nature of $G_q^{(d)}$ and the distinction between the main and the anomaly terms. The dependence of the anomaly on $z$ is regulated by a universal term $(1+t)^z$ that explains the convolution structure of the anomaly sum in \eqref{altzln}.

This structure is quite common for functions appearing in quantum field theory. The main terms usually reflect the expected symmetry of a system while anomalies require additional choices to be defined. For instance, Witten's Chern-Simons path integral is not a $3$-manifold invariant because of the {\it framing anomaly} \cite{At,Wj} but can be turned into one by choosing the canonical $2$-framing to cancel the anomaly \cite{T}. In our case the obvious extra choice is that of a logarithm branch to define $z:=\log_qa$. Disregarding anomalies comes at a price: fixing the canonical $2$-framing would complicate the gluing rules and $G_q^{(d)}$ would not obey a simple functional equation. The Reshetikhin-Turaev normalization may be to blame for nasty prefactors in \eqref{GqZS3} that spoil the duality. It might be of interest to find an analog of $2$-framings on the Gromov-Witten side and compare answers without making choices, even canonical ones. 

From the alternating formula \eqref{altz} we now derive a single product formula for $G_q^{(d)}$ originally given by Nishizawa \cite{Ni1}.
\begin{theorem}[Nishizawa product]\label{Niprod} Let $q\in\D\backslash\R_-$ then quantum multigammas are given by 
\be\label{niprod}
G_q^{(d)}(z+1)=(1-q)^{-\binom{z}{d}}
\prod_{n=1}^\infty\frac{(1-q^n)^{\binom{z-n}{d-1}}}{(1-q^{z+n})^{\binom{-n}{d-1}}}\quad\text{for }d\geq1.
\ee
In particular, for quantum gamma and Barnes functions ($d=1,2$)
\begin{align}
\G_q(z+1) &=(1-q)^{-z}\prod_{n=1}^\infty\,\left(\frac{1-q^{z+n}}{1-q^n}\right)^{-1}\label{qGprod}\\
G_q(z+1) &=(1-q)^{-\frac{z(z-1)}2}\prod_{n=1}^\infty\,\frac{(1-q^{z+n})^n}{(1-q^n)^{-z+n}}\label{qBprod}
\end{align}
\end{theorem}
\begin{proof} Applying \eqref{singprod} to the first product in \eqref{altz} we get
\begin{multline}\label{sumexp}
\prod_{i=0}^d\,(q;q)_\infty^{(i)\ (-1)^{i+1}\binom{z}{d-i}}
=\prod_{i=0}^d\prod_{n=0}^\infty(1-q^{n+1})^{(-1)^{i+1}\,\binom{n+i-1}{n}\,\binom{z}{d-i}}\\
=\prod_{n=0}^\infty(1-q^{n+1})^{\sum_{i=0}^d(-1)^{i+1}\,\binom{n+i-1}{n}\,\binom{z}{d-i}}.
\end{multline}
As shown in the Appendix \eqref{StirBin}, when $d\geq1$
$$
\sum_{i=0}^d\,(-1)^{i+1}\,\binom{n+i-1}{n}\,\binom{z}{d-i}
=\begin{cases}\binom{z-n-1}{d-1}, &\,n\geq1\\ \binom{z-1}{d-1}-\binom{z}{d}, &\,n=0.\end{cases}
$$
Splitting the $n=0$ and $n\geq1$ factors in \eqref{sumexp} we see that it equals
\begin{multline*}
(1-q)^{\binom{z-1}{d-1}-\binom{z}{d}}\ \prod_{n=1}^\infty(1-q^{n+1})^{\binom{z-n-1}{d-1}}
=(1-q)^{-\binom{z}{d}}\ \prod_{n=0}^\infty(1-q^{n+1})^{\binom{z-n-1}{d-1}}\\
=(1-q)^{-\binom{z}{d}}\ \prod_{n=1}^\infty(1-q^{n})^{\binom{z-n}{d-1}}.
\end{multline*}
The second factor in \eqref{altz} can be transformed into a single product using \eqref{singprod} again
$$
\prod_{n=1}^\infty(1-q^{z+n})^{(-1)^{d}\binom{d+n-2}{d-1}}=\prod_{n=1}^\infty(1-q^{z+n})^{\binom{-n}{d-1}},
$$
where we also applied \eqref{Stir-} and \eqref{Stirid} to get the second expression. Multiplying the righthand 
sides of the last two formulas gives the claim.
\end{proof}
Formula \eqref{qGprod} is the standard expression given for the quantum gamma function, see e.g. \cite{AAR}, and \eqref{qBprod} is its closest quantum Barnes analog. Although Nishizawa product \eqref{niprod} may appear prettier than 
the alternating formula \eqref{altz} it completely hides the graded structure of $q$-multigammas.
\smallskip

\noindent {\bf What about $|q|=1$?} By \refT{Altz} $q$-multigammas are {\it entire} functions of $z$ for $|q|<1$. Things 
change if we venture onto the unit circle. As one can judge by the example of the classical Euler's $\G$ 
where $q=1$, \eqref{gqdef} still defines a unique function of $z$ but this time it is only {\it meromorphic}. The appearence of poles prevents infinite products \eqref{niprod} from converging. When $|q|=1$ but {\it $q$ is not a root of unity} $G_q^{(1)}=\G_q$ is constructed explicitly in \cite{NiU} via Shintani's double sine function (see \cite{Wk}). This case is complementary to the classical one $q=1$ and the poles are located at the points $-n-m/\tau$ with $n,m\in\Z$ and $q=e^{2\pi i\tau}$. Clearly, it is desirable to have a similar construction for $G_q^{(2)}=G_q$ especially when $q$ {\it is} a root of unity. Indeed, this case is most relevant to the Chern-Simons theory where $G_q$ is essentially the partition function of $S^3$ by \refT{Comparison} and probably appears as a factor in partition functions of other manifolds.

There is a small window through which we can peak at what happens on the unit circle for any $q$. When $z=N$ is a non-negative integer the product \eqref{niprod} terminates. In the context of Chern-Simons theory $N$ is the rank of $\Sl_N\C$ and we already saw this phenomenon for the quantum Barnes function in \refL{ZqBarnes}. This also generalizes the fact that although $\G(z+1)$ can not be simply expressed as an infinite product for complex $z$ we have nonetheless
$\G(N+1)=1\cdots N$.
\begin{corollary}[Nishizawa]\label{Nniprod}
For a non-negative integer $N$ and any $q\in\C$ one has 
\be\label{nniprod}
G_q^{(d)}(N+1)=(1-q)^{-\binom{N}{d}}\,\prod_{n=1}^N\ (1-q^n)^{\binom{N-n}{d-1}}.
\ee
\end{corollary}
\begin{proof} By the Nishizawa product for $G_q^{(d)}$
\begin{multline*}
(1-q)^{\binom{N}{d}}\,G_q^{(d)}(N+1)
=\frac{\prod_{n=1}^\infty\,(1-q^{n})^{\binom{N-n}{d-1}}}{\prod_{n=1}^\infty\,(1-q^{n+N})^{\binom{-n}{d-1}}}\\
=\frac{\prod_{n=1}^\infty\,(1-q^{n})^{\binom{N-n}{d-1}}}{\prod_{n=N+1}^\infty\,(1-q^{n})^{\binom{-(n-N)}{d-1}}}
=\prod_{n=1}^N\ (1-q^n)^{\binom{N-n}{d-1}}.
\end{multline*}
This gives a proof for $|q|<1$ but since the righthand side is a polynomial in $q$ the general case follows by analytic continuation.
\end{proof}
Note that by \eqref{nniprod} $G_q^{(d)}(N+1)=0$ for $N>n$ if $q$ is an $n$-th root of unity but this never affects Chern-Simons invariants since for them $n=k+N$ with positive integer $k$. This termination phenomenon sheds some light on why so far Chern-Simons invariants have only been defined for integral $z=N$, when it can be interpreted as the rank of a quantum group or an affine Lie algebra. To include complex $z$ one should probably study more involved algebraic/analytic structures. Reversing the perspective, we observe that termination is what makes $G_q$ meaningful in Chern-Simons theory that in its current form produces values only at roots of unity. Note that with the exception of $(1-q)^{-\binom{z}{d}}$ all anomaly terms in \eqref{qBalt} are required for termination. In particular, the ratio $\frac{(aq;q)_\infty^{(2)}}{(q;q)_\infty^{(2)}}$ with $q^{z}=a$ that we recover from Gromov-Witten/Donaldson-Thomas theory, does not suffice for duality to even make sense. It does not terminate for $z=N$ and becomes meaningless for $q=e^{2\pi i/(k+N)}$, which is where the Chern-Simons invariants are defined.

\section*{Conclusions}

Recently the phenomenon of {\it holography} has become prominent in physics \cite{Tr}. Roughly, the idea is that to every theory on a bulk space there corresponds an equivalent theory living on its boundary. The most famous example is the AdS/CFT correspondence of Maldacena but one can certainly trace the analogy back to the classical potential theory, where a harmonic function is recovered from its boundary values. We see a toy example of holography playing out on the {\it ranges} of Calabi-Yau invariants. The master-invariant is a holomorphic function in the bulk (the unit disk) and Donaldson-Thomas theory comes closest to being a bulk theory by giving its Taylor coefficients at $0$. Chern-Simons theory gives its values on the boundary (the unit circle) and Gromov-Witten theory also lives on the boundary but via asymptotic coefficients at $1$. 
Recall that $q=e^{ix}$ and in terms of the string coupling constant $x$ we are dealing with periodic holomorphic functions on the upper half-plane. Their Fourier coefficients are given by Donaldson-Thomas invariants and values at the {\it cusps of $SL_2\Z$} are given by Chern-Simons invariants. This brings to mind the classical {\it modular forms} \cite{Ap} and indeed the relationship between them and the corresponding boundary objects was interpreted recently as an example of the holographic correspondence \cite{MM}. Although our functions are not modular in the classical sense they do have some modular transformation properties \cite{LZg}. This links large $N$ duality to a well-known problem in analytic number theory and perhaps this can aid in its proof.

There are some serious difficulties to be resolved. First, values of a holomorphic function at roots of unity are not enough to recover it uniquely. There are plenty of holomorphic functions on the unit disk that vanish at all roots of unity. Take a modular cusp-form $f$ on the upper half-plane \cite{Ap} for example and consider $f(\ln q/2\pi i)$. The same is true of an asymptotic expansion at a point on the unit circle or even a collection of asymptotic expansions at every root of unity \cite{LZg}. In other words, absent extra data Chern-Simons and Gromov-Witten theories do not recover the master-invariant. Donaldson-Thomas theory has a converse problem: it sure gives a function on the unit disk but one that is singular at each point of the unit circle. For the resolved conifold we reconciled the results by 'completing a pattern' by hand but this will not work in general. A large part of the problem in proving equality of partition functions is that as mathematically defined they are not quite equal.

We contrived to make large $N$ duality work for the resolved conifold. What about other cases? Similarities between the topological vertex and the Reshetikhin-Turaev calculus outlined in Sections \ref{S4},\ref{S5} are promising but there is a vast difference between toric webs and link diagrams representing the dual objects. This ought to be expected of course since not all large $N$ duals to $3$-manifolds are toric and certainly most toric Calabi-Yaus are not dual to any $3$-manifold. However, not all is lost. Link invariants are known to extend to invariants of {\it knotted trivalent graphs} \cite{MO} and these include toric graphs that are also trivalent but planar. Then there is the case of lens spaces (cyclic quotients of $S^3$ \cite{Rf}) where the duals are known and toric \cite{AKMVm,IqK}. Moreover, there exist physical generalizations of the topological vertex to non-toric threefolds \cite{DFS} some of which are dual to non-cyclic spherical quotients.

Even if diagrammatic presentations are reconciled there is yet another hurdle to wrestle with. Topological vertex expressions are infinite sums over all partitions $\P$. Reshetikhin-Turaev expressions are finite sums over partitions in the $(N-1)\times k$ rectangle $\P_{N-1}^k$ and ad hoc tricks are required to turn them into functions of $q=e^{2\pi i/(k+N)}$. Note that the Reshetikhin-Turaev building blocks $W_{\lambda\mu}(N;q)$ are already expressed in a desirable form. It is when the {\it gluing rules} are applied that $q$ has to be specialized to a root of unity to render the sums finite. Analogous infinite sums naively diverge but so would sums of $\calW_{\lambda\mu}(q)$ in the topological vertex without the convergence factors like $(-a)^{|\lambda|}$ in \eqref{z'curW}. After the sum is performed it is a again possible to turn it into a function of $N,q$ up to framing and volume factors. Recall that values at roots of unity do not in themselves determine a holomorphic function on the unit disk appearing in the dualities, and Donaldson-Thomas theory as a bulk theory is incomplete. An enticing possibility is that Chern-Simons theory itself can be turned into a bulk theory.
\begin{quote}\noindent {\bf Conjecture:}\ There exists a {\it universal Chern-Simons TQFT} that assigns holomorphic functions of $q,z$ to links and $3$-manifolds with $q\in\D$. These functions extend to roots of unity in $q$ and specialize to ordinary Chern-Simons invariants at rank $N$ and level $k$ upon specializing to $z=N$, $q=e^{2\pi i/(k+N)}$ up to normalization factors. Partitions serve as colors of link components and the gluing rules are expressed via sums over all partitions. 
Donaldson-Thomas invariants of the dual Calabi-Yau threefolds are Taylor coefficients at $q=0$ of these functions and Gromov-Witten invariants are their asymptotic coefficients at $q=1$.
\end{quote}
This conjecture formalizes the idea that extra extension data for invariants comes from algebraic restrictions required by the axioms of TQFT \cite{At}. A corollary of this conjecture that seems to be within reach is that sums of the type 
$$\sum_{\lambda_1,\dots,\lambda_{n}\in\P_{N-1}^k}
\!\!\!\!\!\!\!W_{\lambda\lambda_1}W_{\lambda_1\lambda_2}\cdots W_{\lambda_n\mu}$$
can be represented as boundary values of sums $$\sum_{\lambda_1,\dots,\lambda_{n}\in\P}
\!\!\!\!\!W_{\lambda\lambda_1}W_{\lambda_1\lambda_2}\cdots W_{\lambda_n\mu}\ r^{|\lambda_1|+\cdots+|\lambda_{n}|},$$ where $r$ is a convergence factor. These and similar sums appear in Chern-Simons invariants of Seifert-fibered $3$-manifolds \cite{Mar1,Roz} and such representation would facilitate a proof of large $N$ duality for them.

\section*{Appendix: Stirling polynomials}

The {\it Stirling polynomials} \cite{Mc},\,I.2.11\, are defined as
\bee
\binom{z}{0}:=1\qquad\binom{z}{n}:=\frac{z(z-1)\cdots(z-n+1)}{n!},\quad n\geq1,
\eee
and reduce to the binomial coefficients when $z=N\geq n$ is a positive integer.
Their generating function is
\be\label{Stirgen}
\sum_{n=0}^\infty\binom{z}{n}\,t^n=(1+t)^z.
\ee
In agreement with \eqref{Stirgen} we always assume $\binom{N}{n}=0$ for $n<0$ or $n>N$. Negative $z$ is also allowed 
\be\label{Stir-}
\binom{-z}{n}=(-1)^n\,\binom{z+n-1}{n}.
\ee
Combining with the generating function \eqref{Stirgen} we have 
\be\label{Stir-gen}
\sum_{n=0}^\infty\binom{z+n-1}{n}\,t^n=(1-t)^{-z}.
\ee
The coefficients $s(n,k)$ in 
\be\label{Stirnum}
\binom{z}{n}=\sum_{k=0}^n\frac{s(n,k)}{n!}\,z^k;\qquad\binom{z+n-1}{n}=\sum_{k=0}^n(-1)^{n-k}\,\frac{s(n,k)}{n!}\,z^k
\ee
are known as the {\it Stirling numbers of the first kind} \cite{Mc},\,I.2.11\,. 
The following identity is easily derived via generating functions
\be\label{Stirid}
\binom{z}{n-1}+\binom{z}{n}=\binom{z+1}{n}.
\ee
For the convenience of the reader, we prove a harder one needed in \refT{Niprod}, namely
\be\label{StirBin}
\sum_{i=0}^d\,(-1)^{i+1}\,\binom{n+i-1}{n}\,\binom{z}{d-i}
=\begin{cases}\binom{z-n-1}{d-1}, &\,n\geq1\\\binom{z-1}{d-1}-\binom{z}{d}, &\,n=0.\end{cases}
\ee
\begin{proof}
Consider the generating function
\begin{multline*}
\sum_{d=0}^\infty\,t^d\left(\sum_{i=0}^d\,(-1)^{i+1}\,\binom{n+i-1}{n}\,\binom{z}{d-i}\right)
=\sum_{i,j=0}^\infty\,t^{i+j}\,(-1)^{i+1}\,\binom{n+i-1}{n}\binom{z}{j}\\
=\sum_{i=0}^\infty\,(-1)^{i+1}\,\binom{n+i-1}{n}\,t^{i}\cdot\sum_{j=0}^\infty\,\binom{z}{j}\,t^{j}.
\end{multline*}
The second factor is obviously $(1+t)^z$ by \eqref{Stirgen}. Assuming $n>0$ so that $\binom{n-1}{n}=0$ the first factor can be rewritten as
$$
t\sum_{i=1}^\infty\,\binom{i-1+(n+1)-1}{i-1}\,(-t)^{i-1}=t(1+t)^{-n-1}
$$
by \eqref{Stir-gen}. Multiplying them and expanding in the powers of $t$
$$
t(1+t)^{z-n-1}=\sum_{d=1}^\infty\binom{z-n-1}{d-1}\,t^d.
$$
We may conclude for $n\geq1$:
$$
\sum_{i=0}^d\,(-1)^{i+1}\,\binom{n+i-1}{n}\,\binom{z}{d-i}
=\begin{cases}\binom{z-n-1}{d-1}, &\,d\geq1\\0, &\,d=0.\end{cases}
$$
The case $n=0$ is special since $\binom{i-1}{0}=1$ by definition. Therefore, 
$$
\sum_{i=0}^\infty\,(-1)^{i+1}\,\binom{i-1}{0}\,t^{i}=-(1+t)^{-1}
$$ and expanding as above we obtain for $n=0$ and all $d$
$$
\sum_{i=0}^d\,(-1)^{i+1}\,\binom{0+i-1}{0}\,\binom{z}{d-i}=-\binom{z-1}{d}.
$$
By assumption we have $d\geq1$ so applying \eqref{Stirid}
$$
-\binom{z-1}{d}=\binom{z-1}{d-1}-\binom{z}{d}.
$$
\end{proof}

\end{document}